\documentclass[reqno]{amsart}
\usepackage{amsmath, amssymb, amsfonts}
\newtheorem{theorem}{Theorem}
\newtheorem{corollary}{Corollary}
\newtheorem{lemma}{Lemma}

\theoremstyle{remark}
\newtheorem{remark}{Remark}
\theoremstyle{definition}
\newtheorem{define}{Definition}
\newtheorem{example}{Example}
\newtheorem*{problem}{Problem}
\newtheorem{claim}{Claim}
\newtheorem*{Acknowlegement}{Acknowlegement}
\newcommand{\CC}{\mathbb C}

\begin{document}

\title[Tangential holomorphic vector fields]{On the existence of tangential holomorphic
vector fields vanishing at an infinite type point}

\author{Ninh Van Thu}

\thanks{ The research of the author was supported in part by an NRF grant 2011-0030044 (SRC-GAIA) of the Ministry of Education, The Republic of Korea.}

\address{Center for Geometry and its Applications,
 Pohang University of Science and Technology,  Pohang 790-784, The Republic of Korea}
\email{thunv@postech.ac.kr, thunv@vnu.edu.vn}

\begin{abstract}
The purpose of this article is to investigate the
holomorphic vector fields tangent to a real hypersurface in $\mathbb C^2$ vanishing at an infinite type point.
\end{abstract}
\maketitle

\section{Introduction}

A {\it holomorphic vector field} in $\mathbb C^n$ takes the form
$$
H = \sum_{k=1}^n h_k (z) \frac{\partial}{\partial z_k}
$$
for some functions $h_1, \ldots, h_n$ holomorphic in $z=(z_1, \ldots, z_n)$.  A smooth real
hypersurface germ $M$ (of real codimension 1) at $p$ in $\mathbb C^n$ takes a defining
function, say $\rho$, such that $M$ is represented by the equation $\rho(z)=0$.  The holomorphic vector field $H$ is said to be {\it tangent} to $M$ if its real part $\hbox{Re }H$ is tangent to $M$, i.e., $H$ satisfies the equation 
\begin{equation}\label{maineq} 
(\hbox{Re }H)\rho(z) = 0~\text{for all}~z\in M.
\end{equation}
We denote by $\mathrm{hol}_0(M,p)$ the vector space of all holomorphic vector fields tangent to $M$ and vanishing at $p$ and by $\mathrm{Aut}(M,p)$ the stability group of $M$, that is, those germs at $p$ of biholomorphisms mapping $M$ into itself and fixing $p$. For the study of $\mathrm{Aut}(M,p)$ and $\mathrm{hol}_0(M,p)$ of various hypersurfaces, we refer the reader to \cite{Bao} and the references therein.

In several complex variables, such tangential holomorphic vector fields arise naturally from the
action by the automorphism group of a domain. If $\Omega$ is a smoothly bounded domain in
$\CC^n$ and if its automorphism group $\mathrm{Aut}(\Omega)$ contains a one-parameter subgroup,
say $\{\varphi_t\}$, i.e., $\varphi_{t+s}=\varphi_t\circ \varphi_s$ for all $t,s\in \mathbb R$ and $\varphi_0=\mathrm{id}_\Omega$, then the $t$-derivative generates a holomorphic vector field. A boundary point $p\in \partial \Omega$ is called a \emph{parabolic orbit accumulation point} (resp. a \emph{hyperbolic orbit accumulation point}) if there is a one-parameter subgroup $\{\varphi_t\}_{t\in \mathbb R}$ of automorphisms such that $\lim_{t\to \pm \infty}\varphi_t(z_0)=p$ (resp. $\lim_{t\to + \infty}\varphi_t(z_0)=p$ and $\lim_{t\to -\infty}\varphi_t(z_0)=q$ for some $\partial \Omega \ni q\ne p$) for some $z_0\in \Omega$. In this circumstance, we call the holomorphic vector field generated by $\{\varphi_t\}_{t\in \mathbb R}$ a \emph{parabolic vector field} (resp. a \emph{hyperbolic vector field}).

In the case when the automorphisms of $\Omega$ extend across the boundary (cf. \cite{Bell-Lig, Fef}),
the vector field generated as such becomes a holomorphic vector field tangent to the boundary hypersurface $\partial\Omega$.
In particular, parabolic and hyperbolic holomorphic vector fields must vanish at their boundary orbit accumulation points. These facts tell us that the study of such vector fields closely pertains to the study of the non-compact automorphism group of $\Omega$, which has been done extensively by many authors (see \cite{IK} and the references therein). Their results, such as the Wong-Rosay theorem \cite{W, R} and the Bedford-Pinchuk-Berteloot theorems \cite{B-P1, B-P2, B-P3, Ber}, depend on the existence of an orbit of an interior point by the action of the automorphism group accumulating at a pseudoconvex boundary point of D'Angelo finite type \cite{D}. For the complementary cases, Greene and Krantz posed a conjecture that for a smoothly
bounded pseudoconvex domain admitting a non-compact automorphism group, the point orbits can
accumulate only at a point of finite type \cite{GK}. The interested reader is referred to the recent papers \cite{IK, Kim-Ninh} for this conjecture.

This paper continues the work that started in \cite{Kim-Ninh} motivated by the following question.
\begin{problem}
 Assume that $(M,p)$ is a non-Leviflat CR hypersurface germ in $\mathbb C^n$ such that $p$ is a point of D'Angelo infinite type.  Characterize all holomorphic vector fields tangent to $M$ vanishing at $p$.
\end{problem}
More precisely, we present a characterization of holomorphic vector fields which are tangent to a $\mathcal{C}^\infty$-smooth hypersurface germ $(M,0)$ of D'Angelo infinite type at the origin $0=(0,0)$ in $\mathbb C^2$ and  vanish at $0$ (cf. Theorems \ref{T1} and \ref{T3} in the next section). As a consequence of our results, any point of D'Angelo infinite type is neither a parabolic nor a hyperbolic orbit accumulation point; this gives a partial answer to the Greene-Krantz conjecture. 

This paper is organized as follows. Two main theorems are stated in Section \ref{S2}. In Section \ref{S3}, we prove Lemma \ref{lemma6} which is a linearization of holomorphic vector fields. Section \ref{S4} is devoted to the proof of Theorem \ref{T1}. In Section \ref{S5}, we introduce the condition $(\mathrm{I})$ and give several examples of functions satisfying the condition $(\mathrm{I})$. The proof of Theorem \ref{T3} is given in Section \ref{S6}. Finally, several technical lemmas are pointed out in Appendix A.

\section{Main results}\label{S2}
For the sake of smooth exposition, we would like to explain the main results of
this article, deferring the proof to the later sections.
\smallskip

Let $M$ be a $\mathcal{C}^\infty$-smooth real hypersurface germ $(M,0)$. Then it admits the following expression:
\begin{equation} \label{eq}
M= \big\{(z_1,z_2)\in \mathbb C^2: \rho(z_1,z_2)
=\mathrm{Re}~z_1+P(z_2)+
(\mathrm{Im}~z_1)Q(z_2,\mathrm{Im}~z_1)=0\big\},
\end{equation}
where $P$ and $Q$ are $C^\infty$-smooth functions with $P(0)=0, dP(0)=0$, and $Q(0,0)=0$. We now discuss what the concept of infinite type means.

Following \cite{D}, we consider a smooth real-valued function $f$ defined in a neighborhood of
$0$ in $\mathbb C$.  Let $\nu_0(f)$ denote the order of vanishing of $f$ at $0$, by the first
nonvanishing degree term in its Taylor expansion at $0$. In the case when $f$ is a mapping into $\mathbb R^k~(k > 1)$, we consider the order of vanishing of all the components and take the smallest one among them for the vanishing order of $f$.  Denote it by
$\nu_0 (f)$. Also denote by $\Delta_r = \{z \in \mathbb C \colon |z|<r\}$ for $r>0$ and by $\Delta:=\Delta_1$. Then the origin is called a \emph{point of D'Angelo infinite type} if, for every integer $\ell > 0$, there exists a holomorphic map
$h:\Delta \to \mathbb C^2$ with $h(0)=(0,0)$ such that
$$
\nu_0 (h) \not=\infty \hbox{ and } \frac{\nu_0 (\rho \circ h)}{\nu_0(h)} > \ell.
$$

We note that if $P$ contains no harmonic terms, then $M$ is of D'Angelo infinite type if and only if $P$ vanishes to infinite order at $0$ (see \cite[Theorem $2$]{Kim-Ninh}). Moreover, in the case that $P(z_2)$ is positive on a punctured disk, K.-T. Kim and the author \cite{Kim-Ninh} showed that there is no non-trivial holomorphic vector field vanishing at the origin tangent to any $\mathcal{C}^\infty$-smooth real hypersurface germ $(M,0)$, except the two following cases:
\begin{itemize}
 \item[(A)] The vanishing order of $Q(z_2,0)$ at $z_2=0$ is finite and $Q(z_2,0)$ contains a monomial term $z_2^k$ for some positive integer $k$.
\item[(B)] The real hypersurface $M$ is rotationally symmetric, i.e., after a change of variable in $z_2$, $\rho(z_1,z_2)=\rho(z_1,|z_2|)$, and in this case the holomorphic vector field is of the form $i\beta z_2 \frac{\partial}{\partial z_2}$ for some non-zero real number $\beta$ (see also \cite{By1}).
\end{itemize}
It is well-known that any rotationally symmetric hypersurface admits non-trivial tangential holomorphic vector fields vanishing at an infinite type point (see also \cite[Theorem $2.1$]{By1}). 

We shall now introduce another class of real hypersurfaces (of course, the case $(\mathrm{A})$ is violated) admiting also non-trivial $\mathrm{hol}_0(M,p)$. Given a nonzero holomorphic function $a(z)=\sum_{n=1}^\infty a_n z^n$ defined on $\Delta_{\epsilon_0}:=\{z\in \mathbb C\colon |z|<\epsilon_0\}~(\epsilon_0>0)$, $\mathcal{C}^\infty$-smooth functions $p,q$ defined respectively on $(0,\epsilon_0)$ and $[0,\epsilon_0)$ satisfying that $q(0)=0$ and that the function
\[
g(z)=
\begin{cases}
e^{p(|z|)}&~\text{if}~0<|z|<\epsilon_0\\
0&~\text{if}~z=0 
\end{cases}
\]
is $\mathcal{C}^\infty$-smooth and vanishes to infinite order at $z=0$, and an $\alpha\in\mathbb R$, we denote by $M(a,\alpha, p,q)$ the germ at $(0,0)$ of a real hypersurface defined by
$$
\rho(z_1,z_2):= \mathrm{Re}~z_1+P(z_2)+F(z_2,\mathrm{Im}~z_1) =0,
$$
where $F$ and $P$ are respectively defined on $\Delta_{\epsilon_0}\times (-\delta_0,\delta_0)$ ($\delta_0>0$ small enough) and $\Delta_{\epsilon_0}$ by
\[
 F(z_2,t)=\begin{cases}
 -\frac{1}{\alpha}\log \Big|\frac{\cos \big(R(z_2)+\alpha t\big)}{\cos (R(z_2))} \Big| &~\text{if}~ \alpha\ne 0\\
 \tan(R(z_2))t  &~\text{if}~ \alpha =0,
\end{cases} 
 \]
 where $R(z_2)=q(|z_2|)- \mathrm{Re}\big(\sum_{n=1}^\infty\frac{a_n}{n} z_2^n\big)$ for all $z_2\in \Delta_{\epsilon_0}$,
and
\begin{equation*}
\begin{split}
  P(z_2)=
 \begin{cases}
\frac{1}{\alpha} \log \Big[ 1+\alpha P_1(z_2)\Big]~&\text{if}~ \alpha \ne 0\\
 P_1(z_2) ~&\text{if}~ \alpha=0,
\end{cases}
\end{split}
\end{equation*}
where 
\begin{equation*}
\begin{split}
P_1(z_2)=\exp\Big[p(|z_2|)+\mathrm{Re}\Big(\sum_{n=1}^\infty  \frac{a_n}{in}z_2^n\Big ) -\log \big|\cos\big(R(z_2)\big)\big|   \Big]
\end{split}
\end{equation*}
for all $z_2\in \Delta_{\epsilon_0}^*$ and $P_1(0)=0$.

Then we can see that $P,F$ are $\mathcal{C}^\infty$-smooth in $\Delta_{\epsilon_0}$ and $P$ vanishes to infinite order at $0$, and hence $M(a,\alpha, p,q)$ is $\mathcal{C}^\infty$-smooth and is of infinite type. 

It follows from \cite[Theorem $3$]{NCM} that the holomorphic vector field
$$
H^{a,\alpha} (z_1,z_2):=L^\alpha (z_1) a(z_2)\frac{\partial }{\partial z_1}+iz_2\frac{\partial }{\partial z_2},
$$
where 
\[
L^\alpha(z_1)=
\begin{cases}
\frac{1}{\alpha}\big(\exp(\alpha z_1)-1\big)&\text{if}~ \alpha \ne 0\\
z_1 &\text{if}~ \alpha =0,
\end{cases}
\]
is tangent to $M(a,\alpha,p,q)$. In addition, $\mathrm{hol}_0\big(M(a,\alpha,p,q),0\big)$ is generated by $H^{a,\alpha}$ (cf. \cite[Corollary $2$]{NCM}) and $\mathrm{Aut}\big(M(a,\alpha,p,q),0\big)$ only consists of the following germs at $0$ of CR automorphisms
\begin{equation*}
\phi^{a,\alpha}_t(z_1,z_2)=
\begin{cases}
\Big(-\frac{1}{\alpha} \log\Big[1+(e^{-\alpha z_1}-1) \exp\big(\int_0^t a(z_2e^{i\tau})d\tau \big)\Big], z_2 e^{it}\Big)&~\text{if}~\alpha\ne 0\\
\Big(z_1 \exp\big(\int_0^t a(z_2e^{i\tau})d\tau \big), z_2 e^{it}\Big)&~\text{if}~ \alpha=0
\end{cases}
\end{equation*}
for all $t\in \mathbb R$, i.e., $\mathrm{Aut}\big(M(a,\alpha,p,q),0\big)$ is the one-parameter group generated by $H^{a,\alpha}$ (cf. \cite[Theorem A]{Ninh1})). 
  
The first aim of this paper is to prove the following theorem, which gives a classification of pairs $(H,M)$ of holomorphic vertor fields $H$ tangent to real hypersurfaces $M$.  
\begin{theorem}\label{T1} If a non-trivial holomorphic vector field germ $(H,0)$ vanishing at the origin is tangent to
a real non-rotationally symmetric hypersurface germ $(M,0)$ defined by the equation
$\rho(z) := \rho(z_1,z_2)=\mathrm{Re}~z_1+P(z_2)+ F(z_2,\mathrm{Im}~z_1)=0$
satisfying the conditions: 
\begin{itemize}
\item[(i)] $F(z_2,t)$ is real-analytic in a neighborhood of $0\in \mathbb C\times \mathbb R$ satisfying $F(z_2,0)\equiv 0$,
\item[(ii)] $P(z_2)>0$ for any $z_2 \not= 0$, and 
\item[(iii)] $P$ vanishes to infinite order at $z_2=0$,
\end{itemize}
then, after a change of variable in $z_2$, $M=M(a,\alpha,p,q)$ and $H=\beta H^{a,\alpha}$ for some nonzero holomorphic function $a$ with $a(0)=0$, $\mathcal{C}^\infty$-smooth real-valued functions $p,q$, and $\beta\in \mathbb R$.
\end{theorem}
\begin{remark}
It is worth noting that the conclusion of Theorem \ref{T1} says that there are no hyperbolic or parabolic orbits of CR automorphisms of $\big(M,0\big)$ accumulating at $0$, since $\phi^{a,\alpha}_t(z)\not \to 0$ as $t\to +\infty$.
\end{remark}
\begin{remark}
As to the hypothesis of the theorem, the condition $\mathrm{(iii)}$ simply tells us that $0$ is a point of infinite type. 
\end{remark}
\begin{remark}
The condition $\mathrm{(i)}$ plays a significant role in the proof of Theorem \ref{T1}. Because of the real-analyticity of $F$, using its power series expansion, each coefficient of $t^k~(k=0,1,\ldots)$ in the equation (\ref{maineq}) imposes some differential equation and therefore our proof follows (cf. Section \ref{S4}). However, in general the function $F$ in the definition of $M(a,\alpha,p,q)$ is not necessarily real-analytic. Moreover, the question of whether there is another $\mathcal{C}^\infty$-smooth real hypersurface of infinite type in $\mathbb C^2$ with non-trivial $\mathrm{hol_0(M,0)}$ remains open.
\end{remark}

We would like to emphasize here that the assumption on the positivity of a function $P$ is essential in the proofs of Theorem \ref{T1} and the main theorems in \cite{Kim-Ninh}. The  following theorem, in which the positivity of a function $P$ is not necessary, is our second main result.
\begin{theorem}\label{T3}
If a $\mathcal{C}^\infty$-smooth hypersurface germ $(M,0)$ is defined by the equation
$\rho(z) := \rho(z_1,z_2)=\mathrm{Re}~z_1+P(z_2)+ (\mathrm{Im}~z_1) Q(z_2,
\mathrm{Im}~z_1)=0$,
satisfying the conditions:
\begin{itemize}
\item[(i)] $P\not \equiv 0$, $P(0)=0$;
\item[(ii)] $P$ satisfies the condition $(\mathrm{I})$ (cf. Definition \ref{def1} in Section \ref{S5});
\item[(iii)] $P$ vanishes to infinite order at $z_2=0$,
\end{itemize}
then any holomorphic vector field vanishing at the origin tangent to $(M,0)$ is identically
zero.
\end{theorem}
\begin{remark}
Theorems \ref{T1} and \ref{T3}, combined with \cite[Theorem $3$]{Kim-Ninh}, are devoted to make a partial answer to the Greene-Krantz conjecture. 
\end{remark}
\begin{remark}[Notations]
Taking the risk of confusion we employ the notations
$$
P'(z) = P_{z} (z) = \frac{\partial P}{\partial z} (z);~ F_{z}(z,t)=\frac{\partial F}{\partial z}(z,t);~F_{t}(z,t)=\frac{\partial F}{\partial t}(z,t)
$$
throughout the paper.  Of course for a function of single real variable $f(t)$, we shall continue using $f'(t)$ for its derivative, as well. In what follows, $\lesssim$ and $\gtrsim$ denote inequalities up to a positive constant multiple. In addition, we use $\approx $ for the combination of $\lesssim$ and $\gtrsim$.
\end{remark}

\section{Linearization of holomorphic vector fields}\label{S3}

Let $b(z)=i\beta z+\cdots$ ($\beta\in \mathbb R^*$) be a holomorphic function on a neighborhood $U$ of the origin. It was proved in \cite{GGJ} that there exists a conformal function $\Phi: V\to U$, where $U$ and $V$ are two open neighborhoods of the origin, such that $\Phi(0)=0$ and $z(t)=\Phi(w_0e^{i\beta t}), -\infty<t<+\infty$, is the solution of the differential equation $\frac{dz(t)}{dt}=b(z(t))=i\beta z(t)+\cdots$ satisfying $z(0)=\Phi(w_0)\in U$. Moreover, one gets 
$$  
\Phi'(w) i\beta w=b(\Phi(w))\;\text{for all}\; w\in V.
$$
The following lemma that will be of use later is a change of variables.
\begin{lemma}\label{lemma6} Let $a,b$ be two holomorphic functions defined on neighborhoods $\Delta_r\times U$ and $U$ of the origins in $\mathbb C^2$ and in $\mathbb C$, respectively, with $b(0)=0$ and $b'(0)=i\beta$, where $\beta\in \mathbb R^*$ and $r>0$. Then,  after the change of variables
$$
 z_1=w_1;z_2=\Phi(w_2),
$$
we obtain that
$$
H(z_1,z_2)= a(z_1,z_2)\frac{\partial }{\partial z_1}+ b(z_2)\frac{\partial }{\partial z_2}
$$
is tangent to the hypersurface
$$ 
M=\big\{(z_1,z_2)\in \Delta_r\times U:\rho(z_1,z_2)= \mathrm{Re}~z_1+  F(z_2,\mathrm{Im}~z_1)=0\big\},
$$
where $F$ is a $\mathcal{C}^1$-smooth function defined on $U\times (-r,r)$, if and only if
$$
\tilde H(w_1,w_2)=a(w_1,\Phi(w_2))\frac{\partial }{\partial z_1}+i\beta w_2\frac{\partial }{\partial z_2}
$$
is tangent to the hypersurface
$$ 
\tilde M=\big\{(w_1,w_2)\in \Delta_r \times V:\tilde\rho(w_1,w_2)= \mathrm{Re}~w_1+F(\Phi(w_2),\mathrm{Im}~w_1)=0\big\}.
$$
\end{lemma}
\begin{proof}
Since $\Phi'(w_2) i\beta w_2=b(\Phi(w_2))$ for all $w_2\in V$, it follows that
$$
i\beta w_2 F_{w_2}(\Phi(w_2),\mathrm{Im}~w_1)= i \beta w_2 \Phi'(w_2)F_{z_2}(z_2,\mathrm{Im}~z_1)=b(z_2)F_{z_2}(z_2,\mathrm{Im}~z_1)   
$$
for all $w_2\in V$. Therefore, we obtain that
\begin{equation*}
\begin{split}
&\mathrm{Re}~H(\rho(z_1,z_2))=\mathrm{Re}\Big[\Big(\frac{1}{2}+F_{z_1}(z_2,\mathrm{Im}~z_1)\Big) a(z_1,z_2)
+F_{z_2}(z_2,\mathrm{Im}~z_1) b(z_2)\Big]\\
&=\mathrm{Re}\Big[\Big(\frac{1}{2}+F_{w_1}(\Phi(w_2),\mathrm{Im}~w_1 )\Big) a(w_1,\Phi(w_2))+F_{w_2}(\Phi(w_2),\mathrm{Im}~w_1) )i \beta w_2\Big]\\
&=\mathrm{Re}~\tilde H(\tilde \rho(w_1,w_2))
\end{split}
\end{equation*}
 for every $(w_1,w_2)\in \Delta_r\times V$, which proves the assertion.
\end{proof}

\section{Proof of Theorem \ref{T1}}\label{S4}

This section is devoted to proving Theorem \ref{T1}. To do this, we divide the proof into six following claims from Claim \ref{c1} to Claim \ref{c6}.

As a first step we shall establish several equations that will be of use later. Let $H(z_1,z_2)=h_1(z_1,z_2)\frac{\partial}{\partial z_1}+h_2(z_1,z_2)\frac{\partial}{\partial z_2}$ and $M$ be a non-trivial holomorphic vector field and a real non-rotationally symmetric hypersurface, respectively, as in Theorem \ref{T1}. Then one has the identity
\begin{equation}\label{eq1221}
(\mathrm{Re}~H) \rho(z)=0,\; \forall z \in M.
\end{equation}

Expand $h_1$ and $h_2$ into the Taylor series at the origin so that
\begin{equation*}
\begin{split}
h_1(z_1,z_2)&=\sum\limits_{j,k=0}^\infty a_{jk} z_1^j z_2^k=\sum\limits_{j=0}^\infty z_1^j a_j(z_2);\\
h_2(z_1,z_2)&=\sum\limits_{j,k=0}^\infty b_{jk} z_1^jz_2^k=\sum\limits_{j=0}^\infty z_1^j b_j(z_2),
\end{split}
\end{equation*}
where $a_{jk}, b_{jk}\in \mathbb C$ and $a_j, b_j$ are holomorphic in a neighborhood of $0\in \mathbb C$ for all $j,k\in \mathbb N$. We note that $a_{00}=b_{00}=0$ since $h_1(0,0)=h_2(0,0)=0$. Moreover, the function $F(z_2, t)$ can be written as
$$
F(z_2,t)=tQ(z_2,t)=\sum_{j=0}^\infty t^{j+1} Q_j(z_2),
$$
where $Q_j~( j=1,2,\ldots)$ are real-analytic in a neighborhood of $0\in \mathbb C$ and $Q(z_2,t):=\sum_{j=0}^\infty t^j Q_j(z_2)$.

By a simple computation, one has 
\begin{equation*}
\begin{split}
 \rho_{z_1}(z_1,z_2)&= \frac{1}{2}+\frac{Q(z_2, \mathrm{Im}~z_1)}{2i}+(\mathrm{Im}~z_1)Q_{z_1}(z_2, \mathrm{Im}~z_1)\\
&= \frac{1}{2}+\frac{Q_0(z_2)}{2i}+\frac{2(\text{Im}~z_1) Q_1(z_2)}{2i}+\frac{3(\text{Im}~z_1)^2 Q_2(z_2)}{2i}+\cdots;\\
\rho_{z_2}(z_1,z_2)&= P'(z_2)+(\text{Im}~z_1) Q_{z_2}(z_2, \text{Im}~z_1),
\end{split}
\end{equation*}
and the equation (\ref{eq1221}) can thus be re-written as
\begin{equation}\label{etn0}
\begin{split}
&\mathrm{Re} \Big[\Big( \frac{1}{2}+\frac{Q(z_2, \mathrm{Im}~z_1)}{2i}+(\mathrm{Im}~z_1)Q_{z_1}(z_2, \mathrm{Im}~z_1)\Big)h_1(z_1,z_2)\\
&\quad +\Big(P'(z_2)+(\text{Im}~z_1)Q_{z_2}(z_2, \text{Im}~z_1)\Big) h_2(z_1,z_2)\Big ]=0
\end{split}
\end{equation}
for all $(z_1,z_2)\in M$.

Since $\Big(it-P(z_2)-tQ(z_2,t), z_2\Big)\in M$ for any $t \in \mathbb R$ with $t$ small enough, the above equation again admits a new form
\begin{equation}\label{eq201311}
\begin{split}
&\mathrm{Re}\Big[ \Big(\frac{1}{2}+\frac{Q_0(z_2)}{2i}+\frac{2t Q_1(z_2)}{2i}+\frac{3t^2 Q_2(z_2)}{2i}+\cdots\Big)\times \\
&\quad\Big(\sum_{j=0}^\infty \big(it-P(z_2)-tQ_0(z_2)-t^2 Q_1(z_2)-\cdots\big)^j a_j(z_2)\Big)\\
&\quad +\Big(P'(z_2)+t {Q_0}_{z_2}(z_2)+t^2 {Q_1}_{z_2}(z_2)+\cdots \Big)\times\\
&\quad \Big( \sum_{m=0}^\infty \big(it-P(z_2)-tQ_0(z_2)-t^2 Q_1(z_2)-\cdots\big)^m b_m(z_2)\Big)\Big]=0
\end{split}
\end{equation}
for all $z_2\in \mathbb C$ and for all $t\in\mathbb R$ with $|z_2|<\epsilon_0$ and $|t|<\delta_0$, where $\epsilon_0>0$ and $\delta_0>0$ are small enough. 

The next step is to demonstrate the following claims. First of all, the following is the first claim, in which its proof only requires the properties $\mathrm{(ii)}$ and $\mathrm{(iii)}$ of the function $P$.
\begin{claim}\label{c1}
 $h_1(0,z_2)\equiv 0$ and $h_2(0,z_2)=i\beta z_2 +\cdots$ for some $\beta\in \mathbb R^*$ and for all $z_2\in\Delta_{\epsilon_0}$. 
\end{claim}
\begin{proof}[Proof of the claim]
Indeed, it follows from (\ref{etn0}) with $t=0$ that 
\begin{equation} \label{e2013}
\begin{split}
\mathrm{Re} \Big [\Big(\frac{1}{2}+\frac{1}{2i}Q(z_2,0)\Big) h_1(0, z_2)\Big ]+O(P(z_2))+O(P'(z_2))=0,~\forall z_2\in\Delta_{\epsilon_0}.
\end{split}
\end{equation}
Because of the fact that $\nu_0(P)=\nu_0(P')=+\infty$, the equation (\ref{e2013}) yields that
$$
\mathrm{Re} \Big [\Big(\frac{1}{2}+\frac{1}{2i}Q(z_2,0)\Big) h_1(0, z_2)\Big ]=0,~\forall z_2\in\Delta_{\epsilon_0}.
$$
Moreover, since $h_1(0,0)=0$ and $Q(0,0)=0$, it is easy to show that the above equation implies that $h_1(0,z_2)\equiv 0$. 

Notice that one may choose $t=\alpha P(z_2)$ in (\ref{etn0})
(with $\alpha \in \mathbb R$ to be chosen later on). Then we get
\begin{equation}\label{eq27}
\begin{split}
\mathrm{Re} &\Big [\Big(\frac{1}{2}+\frac{1}{2i}Q(z_2,\alpha P(z_2))+
\alpha P(z_2)Q_{z_1}(z_2,\alpha P(z_2))\Big)\times \\
& \quad h_1\Big(i\alpha P(z_2)- P(z_2)- \alpha P(z_2)Q(z_2,\alpha P(z_2)), z_2\Big)\\
&\quad +\Big(P'(z_2)+ \alpha  P(z_2) Q_{z_2}(z_2,\alpha  P(z_2))\Big)\times\\
& \quad h_2\Big(i \alpha  P(z_2)- P(z_2)-\alpha  P(z_2)Q(z_2,\alpha P(z_2)),z_2\Big)
 \Big ]=0
\end{split}
\end{equation}
for all $z_2\in \Delta_{\epsilon_0}$.

We remark that if $h_2\equiv 0$, then (\ref{etn0}) shows that $h_1\equiv 0$. Conversely, if $h_1\equiv 0$, then by Lemma \ref{Al5} in Appendix A.3, $M$ is rotational symmetric, which is impossible. So one may assume that $h_1\not\equiv 0$ and $h_2\not\equiv 0$. Let $j_0$ be the smallest integer such that $a_{j_0 k}\ne 0$ for some integer $k$. Then let $k_0$ be the smallest integer such that $a_{j_0 k_0}\ne 0$. Similarly, let $m_0$ be the smallest integer such that $b_{m_0 n}\ne 0$ for some integer $n$. Then let $n_0$ be the smallest integer such that $b_{m_0 n_0}\ne 0$. Note that $j_0\geq 1$ since $h_1(0,z_2)\equiv 0$.

 Since $P(z_2)=o(|z_2|^{n_0})$, it follows from (\ref{eq27}) that
\begin{equation}\label{etn2}
\begin{split}
\mathrm{Re} &\Big[\frac{1}{2} a_{j_0k_0}(i\alpha -1)^{j_0}(P(z_2))^{j_0}z_2^{k_0}+
 b_{m_0n_0}(i\alpha -1)^{m_0}\big(z_2^{n_0}+o(|z_2|^{n_0})\big) \\
&\times (P(z_2))^{m_0} \Big(P'(z_2)+\alpha P(z_2)Q_{z_2}(z_2, \alpha P(z_2))\Big)   \Big ]=o(P(z_2)^{j_0}|z_2|^{k_0})
\end{split}
 \end{equation}
for all $|z_2|<\epsilon_0$ and for all $\alpha \in \mathbb R$ small enough. We note that in the case $k_0=0$ and $\mathrm{Re}(a_{j_0 0})=0$, $\alpha$ can be chosen in such a way that $\mathrm{Re}\big( (i\alpha-1)^{j_0}a_{j_0 0}\big)\ne 0$. Then the above equation yields that $j_0>m_0$. We conclude from Lemma \ref{Al3} in Appendix A.3 that $m_0=0,n_0=1$, and $b_{0,1}=i\beta z_2$ for some $\beta\in \mathbb R^*$. Therefore, the claim is proved.
\end{proof}

Now by a change of variables as in Lemma \ref{lemma6}, without loss of generality we may assume that $b_0(z_2)=i\beta z_2$. Moreover, we have the following claims.

\begin{claim}\label{c2} One has that $a_1(z_2)=\beta\sum_{n=1}^\infty a_n z_2^n\not \equiv 0$ and 
\begin{equation*}
\begin{split}
Q_0(z_2) &=\tan(R(z_2)) ;\\
P(z_2)&=\exp\Big[ p(|z_2|)+\mathrm{Re}\Big(\sum_{n=1}^\infty  \frac{a_n}{in}z_2^n\Big )-\log \big|\cos\big(R(z_2)\big )\big|+ v(z_2)\Big]
\end{split}
 \end{equation*}
for all $z_2\in \Delta_{\epsilon_0}^*$, where $R(z_2)=q(|z_2|)- \mathrm{Re}\Big(\sum_{n=1}^\infty \frac{a_n}{n} z_2^n\Big)$ for all $z_2\in \Delta_{\epsilon_0}^*$, $v$ is a $\mathcal{C}^\infty$-smooth function on $\Delta_{\epsilon_0}$ with $\nu_0(v)=+\infty$, and $q,p$ are $\mathcal{C}^\infty$-smooth functions on $(0,\epsilon_0)$ and are chosen so that $R$ is real-analytic in $\Delta_{\epsilon_0}$ and  that $P$ is $\mathcal{C}^\infty$-smooth in $\Delta_{\epsilon_0}$ with $\nu_0(P)=+\infty$.
\end{claim}
\begin{proof}[Proof of the claim]
First of all, taking $\frac{\partial}{\partial t}$ of both sides of the equation (\ref{eq201311}) at $t=0$, we obtain that
\begin{equation}\label{eq201313}
\begin{split}
&\mathrm{Re}\Big \{P'(z_2)\big(i-Q_0(z_2)\big)\Big[ b_1(z_2)+2(-P(z_2))b_2(z_2)+\cdots\\
&\quad +m(-P(z_2))^{m-1}b_m(z_2)+\cdots\Big]\\
&\quad +\frac{i}{2}\Big(1+Q_0^2(z_2)\Big)\Big[ a_1(z_2)+2(-P(z_2))a_2(z_2)+\cdots+m(-P(z_2))^{m-1}a_m(z_2)+\cdots\Big]\\
& \quad +{Q_0}_{z_2}\Big[ i\beta z_2+(-P(z_2))b_1(z_2)+\cdots+(-P(z_2))^{m}b_m(z_2)+\cdots\Big]\\
&\quad +\frac{Q_1(z_2)}{i}\Big[(-P(z_2))a_1(z_2)+(-P(z_2))^2 a_2(z_2)+\cdots\\
&\quad +(-P(z_2))^{m}a_m(z_2)+\cdots\Big] \Big\}=0
\end{split}
\end{equation}
for all $z_2\in \Delta_{\epsilon_0}$. Since $Q_0$ is real-analytic and $\nu_0(P)=\nu_0(P')=0$, one gets
\begin{equation}\label{eq2014}
\mathrm{Re}\Big[2i\beta  z_2 {Q_0}_{z_2}(z_2)+i a_1(z_2)\Big(1+Q_0^2(z_2)\Big)\Big]\equiv 0
\end{equation}
on $\Delta_{\epsilon_0}$. We note that the equation (\ref{eq2014}) shows that $\mathrm{Re}(ia_1(0))=0$.

Therefore, the solution $Q_0$ of Eq. (\ref{eq2014}) has the form as in the claim  (following the proof of Lemma \ref{pde1} in Appendix A.2).  In addition, since the real hypersuface $M$ is not rotationally symmetric, by \cite[Theorem 3]{Kim-Ninh} mentioned as in Section \ref{S2}, $Q_0$ must contain a monomial term $z_2^k$ for some positive integer $k$. Consequently, we have in fact that $a_1\not \equiv 0$.

Next, it follows from (\ref{eq201311}) with $t=0$ that
\begin{equation}\label{eq201316}
\begin{split}
&\mathrm{Re}\Big[-\Big( \frac{1}{2}+\frac{Q_0(z_2)}{2i}\Big)a_1(z_2)P(z_2) +  i\beta z_2 P'(z_2) \Big] +O(P(z_2)^2)+O(P'(z_2)P(z_2))=0,
\end{split}
\end{equation}
or equivalently
\begin{equation}\label{eq2013166}
\begin{split}
2\mathrm{Re}\Big(i\beta z_2 \frac{P_{z_2}(z_2)}{P(z_2)}\Big)=\mathrm{Re}\big( a_1(z_2)\big) +Q_0(z_2)\mathrm{Re}\big(\frac{a_1(z_2)}{i}\big)+ O(P(z_2))+O(P'(z_2))
\end{split}
\end{equation}
for every $z_2\in \Delta^*_{\epsilon_0}$. By \cite[Lemma 1]{Kim-Ninh}, it follows from Eq. (\ref{eq2013166}) that $\mathrm{Re}(a_1(0))=0$, which, together with the above-mentioned fact that $\mathrm{Re}(ia_1(0))=0$, shows that $a_1(0)=0$. 

Now the solution $P$ of Eq. (\ref{eq2013166}) has the form as claimed (following the proof of Lemma  \ref{pde1} in Appendix A.2). Therefore, this completes the proof.
\end{proof}
 We now observe that $\limsup_{r\to 0^+} |r p'(r)|=+\infty$, for otherwise one gets $|p(r)|\lesssim |\log(r)|$ for every $0<r<\epsilon_0$, and thus $P$ does not vanish to infinite order at $0$. Furthermore, a direct calculation shows that 
\begin{equation} \label{eq201318}
\begin{split}
z_2 \frac{P_{z_2}(z_2)}{P(z_2)} =\frac{1}{2}|z_2|p'(|z_2|)+ g(z_2) 
\end{split}
\end{equation}
for all $z_2\in \Delta_{\epsilon_0}$, where $g\in \mathcal{C}^\infty(\Delta_{\epsilon_0})$.
\begin{claim}\label{c3} $b_1\equiv 0$ on $\Delta_{\epsilon_0}$.
\end{claim}
\begin{proof}[Proof of the claim]
To obtain a contradiction, we suppose that $b_1\not \equiv 0$, it follows from (\ref{eq201313}) and (\ref{eq2014}) that
\begin{equation}\label{equation4}
\begin{split}
&\mathrm{Re}\Big \{\Big( (i-Q_0(z_2))b_1(z_2)\Big)\frac{P'(z_2)}{P(z_2)}-i a_2(z_2)\Big(1+Q_0^2(z_2)\Big)\\
&\quad -{Q_0}_{z_2}(z_2)b_1(z_2)-\frac{Q_1(z_2)}{i}a_1(z_2)+O(P(z_2))+O(P'(z_2)) \Big\}\equiv 0
\end{split}
\end{equation}
on $\Delta_{\epsilon_0}$. We will show that $b_1(z_2)=\tilde \beta z_2+\cdots$ for some $\tilde \beta\in \mathbb R^*$. To prove this, we consider the following cases.

\smallskip

\noindent
{\bf Case 1.} {\boldmath $b_1(0)\ne 0$.} In this case, let $\gamma: (-1,1) \to \Delta_{\epsilon_0} \subset\mathbb C$ be a $\mathcal{C}^\infty$-smooth curve such that $\gamma'(t)=\big(i-Q_0(\gamma(t))\big)b_1(\gamma(t))$ for all $|t|<1$ and $\gamma(0)=0$. It follows from (\ref{equation4}) that $\mathrm{Re}\Big(\big(i-Q_0(z_2)\big)b_1(z_2) P_{z_2}(z_2)/P(z_2)\Big)$ is bounded on $\Delta^*_{\epsilon_0}$, and thus 
$$
\frac{d}{dt}\log P(\gamma(t))=2\mathrm{Re}\Big(\gamma'(t)P_{z_2}(\gamma(t))/P(\gamma(t))\Big)
$$
is also bounded on $(-1,1)$. This implies that $\log P(\gamma(t))=O(t)$, which contradicts the fact that $P(\gamma(t))\to 0$ as $t\to 0$. Therefore, we conclude that $b_1(0)=0$.

\smallskip

\noindent
{\bf Case 2.} {\boldmath ${b_1}'(0)\not \in \mathbb R^*$.}  It follows from (\ref{eq201318}) and (\ref{equation4}) that  
$$
u(z_2):=\mathrm{Re}\Big(\big(i-Q_0(z_2)\big)\tilde b_1(z_2)|z_2| p'(|z_2|)\Big)-\tilde g(z_2)=0
$$
for all $z_2\in \Delta_{\epsilon_0}$, where $\tilde g(z_2)$ is a $\mathcal{C}^\infty$-smooth real-valued function defined on $\Delta_{\epsilon_0}$ and $\tilde b_1(z_2):=b_1(z_2)/z_2$ if $z_2\ne 0$ and $\tilde b_1(0)=b_1'(0)$.

Since $\limsup_{r\to 0^+} r|p'(r)|= +\infty$, it follows that the function $\tilde g(z_2)\not \equiv 0$ and vanishes to finite order at $z_2=0$. It can therefore be written as $\tilde g(z_2)=\sum_{0\leq j\leq l} g_j z_2^{l-j}\bar z_2^j+o(|z_2|^l)$ with $g_j\in\mathbb C $ and $g_j=\overline{g_{l-j}}$, where $l=\nu_0(\tilde g)$. Because $\limsup_{r\to 0^+}| rp'(r)|=+\infty$, we have $m:=\nu_0(\tilde b_1)>l$, and thus by taking $\limsup_{r\to 0^+} \frac{1}{r^{l}}u(r e^{i\theta})$  for each $\theta\in\mathbb R$ we obtain that
$$
\cos\big(m\theta+\varphi)= \sum_{0\leq j\leq l} g_j e^{i (l-2j)\theta }
$$
for all $\theta\in\mathbb R$, where $\varphi $ is a real number. This implies that the functions $1, \cos (\theta)$, $\sin (\theta),\ldots,\cos (m\theta),\sin (m\theta)$ are linearly dependent, which leads to a contradiction. 

Altogether, we conclude that $b_1(z_2)=\tilde\beta z_2+\cdots=\tilde \beta z_2(1+O(z_2))$ for some $\tilde \beta \in \mathbb R^*$. Furthermore, from (\ref{eq2013166}) and (\ref{equation4}) we have that  
\begin{equation}\label{equation5}
\begin{split}
&\mathrm{Re}\Big \{\Big( (i-Q_0(z_2))b_1(z_2)-i\tilde \beta z_2\Big)\frac{P'(z_2)}{P(z_2)}\Big\}-\Big(1+Q_0^2(z_2)\Big)\mathrm{Re}\big(ia_2(z_2)\big)\\
&\quad -\mathrm{Re}\big({Q_0}_{z_2}(z_2)b_1(z_2)\big)-\mathrm{Re}\big(\frac{Q_1(z_2)}{i}a_1(z_2)\big)\\
&\quad -\frac{\tilde \beta}{2\beta}\Big(\mathrm{Re}\big( a_1(z_2)\big) +Q_0(z_2)\mathrm{Re}\big(\frac{a_1(z_2)}{i}\big)\Big)+ O(P(z_2))+O(P'(z_2))\equiv 0 
\end{split}
\end{equation}
on $\Delta_{\epsilon_0}^*$. Let us denote by $c(z_2)$ the real-analytic function on $\Delta_{\epsilon_0}$ defined by 
$$
c(z_2):=\frac{(i-Q_0(z_2))b_1(z_2)-i\tilde \beta z_2}{z_2}
$$
for all $z_2\in \Delta_{\epsilon_0}^*$. Since $Q_0$ contains  non-harmonic terms, $\mathrm{Re}(c(z_2))\not \equiv 0$. Moreover, by (\ref{eq201318}) and (\ref{equation5}) the function $\mathrm{Re}\big(c(z_2)\big)|z_2| p'(|z_2|)$ extends to be $\mathcal{C}^\infty$-smooth in $\Delta_{\epsilon_0}$.

We now prove that there exist $c>0$ and $n\in \mathbb N^*$ such that $p(r)=-\frac{c}{r^n}(1+\gamma(r))$ for all $0<r<\epsilon_0$, where $\gamma: [0,\epsilon_0)\to \mathbb R$ is $\mathcal{C}^\infty$-smooth and satisfies $\gamma(r)\to 0$ as $r\to 0$. Indeed, suppose otherwise. Then the function $\mathrm{Re}\big(c(z_2)\big)|z_2| p'(|z_2|)$ cannot extend to be $\mathcal{C^\infty}$-smooth in $\Delta_{\epsilon_0}$ since $\limsup_{r\to 0^+}r |p'(r)|=+\infty$ and $p(r)\not \approx -\frac{1}{r^m}$ for any $m\in \mathbb N^*$, which is a contradiction. Thus, the assertion is proved.

 We note that Eq. (\ref{eq201311}) with $t=0$ implies that 
\begin{equation}\label{eq20137181}
\begin{split}
&\mathrm{Re}\Big[-\Big( \frac{1}{2}+\frac{Q_0(z_2)}{2i}\Big)a_1(z_2) +i\beta z_2\frac{P'(z_2)}{P(z_2)}-b_1(z_2) P'(z_2) \Big] +O(P(z_2)))=0
\end{split}
\end{equation}
for all $z_2\in\Delta^*_{\epsilon_0}$. By Claim \ref{c2}, we have that
\begin{equation*}
\begin{split}
P(z_2)&=\exp\Big[ p(|z_2|)+\mathrm{Re}\Big(\sum_{n=1}^\infty  \frac{a_n}{in}z_2^n\Big )-\log \big|\cos\big(R(z_2)\big )\big|+ v(z_2)\Big],
\end{split}
 \end{equation*}
where $v\in \mathcal{C}^\infty(\Delta_{\epsilon_0})$. Moreover, a simple computation shows that
\begin{equation}\label{eq20137182}
\begin{split}
2\mathrm{Re}\Big(i\beta z_2 \frac{P_{z_2}(z_2)}{P(z_2)}\Big)=\mathrm{Re}\big( a_1(z_2)\big) +Q_0(z_2)\mathrm{Re}\big(\frac{a_1(z_2)}{i}\big)+2\mathrm{Re}\big(i\beta z_2 v_{z_2}(z_2)\big)
\end{split}
\end{equation}
for every $z_2\in \Delta^*_{\epsilon_0}$ and that
\begin{equation}\label{eq20137183}
\begin{split}
2\mathrm{Re}\Big(b_1(z_2)P_{z_2}(z_2)\Big)=2\mathrm{Re}\Big(\tilde \beta z_2 \big(1+O(z_2)\big)P_{z_2}(z_2)\Big)=n\tilde \beta c \frac{1}{|z_2|^n} \big(1+O(|z_2|)\big)P(z_2)
\end{split}
\end{equation}
for every $z_2\in \Delta^*_{\epsilon_0}$. Therefore, it follows from  (\ref{eq20137181}), (\ref{eq20137182}), and (\ref{eq20137183}) that 
\begin{equation}\label{eq20137184}
\begin{split}
2\mathrm{Re}\big(iz_2 v_{z_2}(z_2)\big)=nc\frac{\tilde \beta}{\beta} \frac{1}{|z_2|^n} \big(1+\tilde \gamma(z_2)\big)P(z_2)
\end{split}
\end{equation}
for every $z_2\in \Delta^*_{\epsilon_0}$, where $\tilde \gamma: \Delta_{\epsilon_0}\to \mathbb R$ is $ \mathcal{C}^\infty$-smooth and $\tilde{\gamma}(z_2)\to 0$ as $z_2\to 0$.

Choose $r\in (0,\epsilon_0)$ such that $\max_{|z_2|=r}|\tilde \gamma(z_2)|\leq\frac{1}{2}$ and $P(re^{it})=\exp\big(p(r)+O(r)\big)=\exp\big(-\frac{c}{r^n}(1+\gamma(re^{it}))+O(r)\big)\geq \exp\big(-\frac{2c}{r^n}\big)$ for all $0\leq t\leq 2\pi$. Let $u(t):=v(re^{it})$ for all $t\in \mathbb R$. Then by (\ref{eq20137184}), one gets
$$
u'(t)=nc\frac{\tilde \beta}{\beta} \frac{1}{r^n} \big(1+\tilde \gamma(re^{it})\big)P(re^{it})
$$
for all $t\in \mathbb R$. Thus, we obtain that
\begin{equation*}
\begin{split}
0=|u(2\pi)-u(0)|&=nc\frac{\tilde \beta}{\beta} \left|\int_{0}^{2\pi} \frac{P(re^{it})}{r^n}\big(1+\tilde \gamma(re^{it})\big) dt\right|\\
                &\geq nc\frac{\tilde \beta}{\beta}\int_{0}^{2\pi} \frac{P(re^{it})}{r^n}\big(1-|\tilde \gamma(re^{it})|\big) dt \geq  nc \frac{\tilde \beta}{2\beta}\int_{0}^{2\pi} \frac{e^{-\frac{2c}{r^n}}}{r^n} dt\\
                &=nc\pi\frac{\tilde \beta}{\beta}\frac{e^{-\frac{2c}{r^n}}}{r^n} >0,
\end{split}
\end{equation*}
which is impossible, and hence our claim is proved.
\end{proof}

\begin{claim}\label{c4} $a_2(z_2)\equiv Q_1(0)a_1(z_2)$ and $Q_1(z_2)\equiv Q_1(0)\Big(1+Q_0^2(z_2)\Big)$ on $\Delta_{\epsilon_0}$.
\end{claim}
\begin{proof}[Proof of the claim]
Since $b_1\equiv 0$ (cf. Claim \ref{c3}), by (\ref{equation4}) and note that $Q_0, Q_1$ are real-analytic, and $P(z_2),P'(z_2)$ vanish to infinite order at $0$, one has
\begin{equation}\label{equation6}
\mathrm{Re}\Big[i \Big(1+Q_0^2(z_2)\Big)a_2(z_2)- i Q_1(z_2) a_1(z_2) \Big]\equiv 0
\end{equation}
on $\Delta_{\epsilon_0}$.

On the other hand, taking $\frac{\partial^2}{\partial t^2}$ of both sides of Eq. (\ref{eq201311}) at $t=0$, we have that
\begin{equation}\label{q1}
\begin{split}
&\mathrm{Re}\Big\{\frac{3Q_2(z_2)}{2i} \Big(-P(z_2) a_1(z_2)+P(z_2)^2 a_2(z_2)+\cdots +(-P(z_2))^{m-1} a_{m-1}(z_2)+\cdots \Big) \\
&\quad+ \frac{Q_1(z_2)}{i} \big(i-Q_0(z_2)\big)\Big( a_1(z_2)-2P(z_2) a_2(z_2)+\cdots\\
&\quad + m (-P(z_2))^{m-1} a_{m}(z_2)+\cdots \Big)+\Big(\frac{1}{2}+\frac{Q_0(z_2)}{2i}\Big)\\
&\quad \times \Big(-Q_1(z_2) a_1(z_2)+\Big[(i-Q_0(z_2))^2+ 2P(z_2)Q_1(z_2)\Big]a_2(z_2)+\cdots \\
&\quad + \Big[\frac{(m+1)m}{2}(-P(z_2))^{m-1}(i-Q_0(z_2))^2 -(m+1)(-P(z_2))^mQ_1(z_2)\Big]a_{m+1}(z_2)\\
&\quad +\cdots\Big)+(Q_0)_{z_2}(z_2)\big(i-Q_0(z_2)\big)  \Big( b_1(z_2)-2 P(z_2) b_2(z_2)+\cdots\\
&\quad + m(-P(z_2))^{m-1}b_m(z_2)+\cdots \Big)\\
&\quad+ (Q_1)_{z_2}(z_2)  \Big(i\beta z_2-P(z_2) b_1(z_2)+\cdots+(-P(z_2))^m b_m(z_2)+\cdots \Big)\\
&\quad+P'(z_2)\Big(-Q_1(z_2) b_1(z_2)+\Big[(i-Q_0(z_2))^2+ 2P(z_2)Q_1(z_2)\Big]b_2(z_2)+\cdots\\
&\quad +  \Big[\frac{m(m-1)}{2}(-P(z_2))^{m-2}(i-Q_0(z_2))^2 -m(-P(z_2))^{m-1}Q_1(z_2)\Big]b_{m}(z_2)\\
&\quad +\cdots\Big)\Big\}\equiv 0~\text{on}~  \Delta_{\epsilon_0}.
\end{split}
\end{equation}
Since $Q_0, Q_1$ are real-analytic, $\nu_0(P)=\nu_0(P')=+\infty$, and $b_1\equiv 0$, we deduce that
\begin{equation}\label{q2}
\begin{split}
&\mathrm{Re}\Big\{i\beta z_2(Q_1)_{z_2}(z_2)+\frac{Q_1(z_2)}{i} (i-Q_0(z_2)) a_1(z_2)+\Big(\frac{1}{2}+\frac{Q_0(z_2)}{2i}\Big)\\
&\quad \times\Big(-Q_1(z_2) a_1(z_2)+(i-Q_0(z_2))^2a_2(z_2)\Big)\Big\}\equiv 0
\end{split}
\end{equation}
on $\Delta_{\epsilon_0}$. This equation implies that $\mathrm{Re}(a_2(0))=0$.
Moreover, Eq. (\ref{equation6}) shows that $\mathrm{Re}(ia_2(0))=0$. Thus $a_2(0)=0$. 

Now the equations (\ref{eq2014}), (\ref{equation6}), and (\ref{q2}) yield the proof of the claim (see Lemma \ref{pde2} in Appendix A.2).
\end{proof}
\begin{claim} \label{c5} One has that
$a_m(z_2)\equiv  \frac{2^{m-1}}{m!}Q_1^{m-1}(0) a_1(z_2) $ and $b_{m-1}(z_2)\equiv 0$ on $\Delta_{\epsilon_0}$ for all $m\geq 2$.
\end{claim}

\begin{proof}[Proof of the claim]
We shall prove the claim by  induction on $m$. For $m=2$, it follows from Claim \ref{c4} and Claim \ref{c3} that $a_2(z_2)\equiv Q_1(0) a_1(z_2)$ and $b_1(z_2)\equiv 0$, respectively. Suppose that $a_2(z_2)\equiv Q_1(0) a_1(z_2), \ldots , a_{m}(z_2)\equiv \frac{2^{m-1}}{m!}Q_1^{m-1}(0) a_1(z_2)$, $b_1(z_2)\equiv \cdots\equiv b_{m-1}(z_2)\equiv 0$ for $m\geq 2$. We will show that $b_{m}(z_2)\equiv 0$ and $a_{m+1}(z_2) \equiv \frac{2^{m}}{(m+1)!}Q_1^{m}(0) a_1(z_2)$.

Indeed, by (\ref{eq201313}) we have
\begin{equation}\label{equation7}
\begin{split}
&\mathrm{Re}\Big \{(-1)^{m-1} m (i-Q_0(z_2))b_m(z_2)\frac{P'(z_2)}{P(z_2)}+(-1)^{m}(m+1)\frac{i}{2}\Big(1+Q_0^2(z_2)\Big)a_{m+1}(z_2)\\
& \quad+(-1)^m b_m(z_2){Q_0}_{z_2}(z_2)+(-1)^m\frac{Q_1(z_2)}{i}  a_m(z_2)+O(P(z_2))+O(P'(z_2)) \Big\}\equiv 0
\end{split}
\end{equation}
on $\Delta_{\epsilon_0}$.

Repeating the argument as in the proof of Claim \ref{c3}, we deduce that $b_m(z_2)\equiv 0$. Thus we obtain that
\begin{equation}\label{qt1}
\begin{split}
&\mathrm{Re}\Big \{(-1)^{m}(m+1)\frac{i}{2}\Big(1+Q_0^2(z_2)\Big)a_{m+1}(z_2)+(-1)^m\frac{Q_1(z_2)}{i}  a_m(z_2)\Big\}\equiv 0.
\end{split}
\end{equation}
$\Delta_{\epsilon_0}$. Consequently, one has $\mathrm{Re}(i a_{m+1}(0))=0$.

On the other hand, since $Q_0,Q_1,Q_2$ are real-analytic, $\nu_0(P)=\nu_0(P')=+\infty$, and $b_1(z_2)\equiv \cdots \equiv b_m(z_2)\equiv 0$, from (\ref{q1}) we have 
 \begin{equation}\label{qt2}
\begin{split}
&\mathrm{Re}\Big \{\frac{3Q_2(z_2)}{2i} a_{m-1}(z_2) +m\frac{Q_1(z_2)}{i} (i-Q_0(z_2))a_m(z_2)\\
&\quad + \Big(\frac{1}{2}+\frac{Q_0(z_2)}{2i}\Big)\Big(\frac{m(m+1)}{2}(i-Q_0(z_2))^2a_{m+1}(z_2)-mQ_1(z_2)a_m(z_2)\Big)
\Big\}\equiv 0
\end{split}
\end{equation}
on $\Delta_{\epsilon_0}$. This implies that $\mathrm{Re}(a_{m+1}(0))=0$, which, together with $\mathrm{Re}(i a_{m+1}(0))=0$ as above, indicates that $a_{m+1}(0)=0$.

Furthermore, since $Q_1(z_2)\equiv Q_1(0)\Big(1+Q_0^2(z_2)\Big)$ (cf. Claim \ref{c4}), we conclude from (\ref{qt1}) that
$$
a_{m+1}(z_2)\equiv \frac{2}{m+1}Q_1(0) a_m(z_2)\equiv \cdots\equiv \frac{2^m}{(m+1)!} Q_1^m(0)a_1(z_2),
$$
as claimed. 
\end{proof}
\begin{claim}\label{c6} One has that
\begin{itemize}
\item[(a)] \[
 F(z_2,t)=\begin{cases}
 -\frac{1}{2Q_1(0)}\log \Big|\frac{\cos \big(R(z_2)+2Q_1(0)t\big)}{\cos (R(z_2))} \Big| &~\text{if}~ Q_1(0)\ne 0\\
 \tan(R(z_2))t  &~\text{if}~ Q_1(0)=0
\end{cases} 
 \]
for all $(z_2,t)\in \Delta_{\epsilon_0}\times (-\delta_0,\delta_0)$, where $R$ is given in Claim \ref{c5}.
\item[(b)]
\[ P(z_2)=
 \begin{cases}
\frac{1}{2Q_1(0)} \log \Big[ 1+2Q_1(0)P_1(z_2)\Big]~&\text{if}~ Q_1(0)\ne 0\\
 P_1(z_2) ~&\text{if}~ Q_1(0)=0
\end{cases}
\]
for all $z_2\in \Delta_{\epsilon_0}$, where 
 \begin{equation*}
\begin{split}
P_1(z_2)=\exp\Big(p(|z_2|)+\mathrm{Re}\Big(\sum_{n=1}^\infty  \frac{a_n}{in}z_2^n\Big ) -\log \big|\cos\big(R(z_2)\big )\big|   \Big)
\end{split}
\end{equation*}
for all $z_2\in \Delta_{\epsilon_0}^*$ and $P_1(0)=0$, where $p,q$ are the functions given in Claim \ref{c2}.
\end{itemize}
\end{claim}
\begin{proof}[Proof of the claim]
By Claim \ref{c5}, it is easy to check that $h_1(z_1,z_2)=z_1a_1(z_2)$ if $Q_1(0)=0$ and 
$$h_1(z_1,z_2)=\frac{1}{2Q_1(0)}\Big[\exp \Big(2Q_1(0)z_1\Big)-1\Big]a_1(z_2)$$ if $Q_1(0)\ne 0$ and $h_2(z_1,z_2)=i\beta z_2$. 

Now we divide the proof into the two following cases.

\smallskip
\noindent
{\bf Case A.~}{\boldmath $Q_1(0)=0$.} From Eq. (\ref{eq201311}) we have that
\begin{equation}\label{eq20141111}
\begin{split}
&\mathrm{Re}\Big\{ \Big(\frac{1}{2}+\frac{Q_0(z_2)}{2i}+\frac{2t Q_1(z_2)}{2i}+\frac{3t^2 Q_2(z_2)}{2i}+\cdots\Big) \\
&\quad\times \Big(it-P(z_2)-tQ_0(z_2)-t^2Q_1(z_2)-\cdots \Big)a_1(z_2)\\
&\quad +\Big(P'(z_2)+t {Q_0}_{z_2}(z_2)+t^2 {Q_1}_{z_2}(z_2)+\cdots \Big) i\beta z_2\Big\}=0
\end{split}
\end{equation}
for all $z_2\in \mathbb C$ and for all $t\in\mathbb R$ with $|z_0|<\epsilon_0$ and $|t|<\delta_0$. Then Eq. (\ref{eq20141111}) with $t=0$ implies easily that 
\begin{equation}\label{eqabc}
\begin{split}
&\mathrm{Re}\Big\{i\beta z_2 P'(z_2)-\Big(\frac{1}{2}+\frac{Q_0(z_2)}{2i}\Big)P(z_2)a_1(z_2) \Big\}\equiv 0
\end{split}
\end{equation}
on $\Delta_{\epsilon_0}$. Therefore, by Lemma \ref{pde1} in Appendix A.2 the function $P(z_2)\equiv P_1(z_2)$, as desired.

Now by Claim \ref{c4}, it follows that $Q_1\equiv 0$, and thus taking $\frac{\partial^2}{\partial t^2}$ of both sides of (\ref{eq20141111}) at $t=0$, we obtain that 
$$
\mathrm{Re}\Big(\frac{3Q_2(z_2)}{2i}(-P(z_2)) a_1(z_2)\Big)\equiv 0
$$
on $\Delta_{\epsilon_0}$. This implies that $Q_2\equiv 0$. Taking $\frac{\partial^m}{\partial t^m}$ of both sides of (\ref{eq20141111}) at $t=0$ for $m=3,\ldots$, we obtain, by induction on $m$, that $Q_m\equiv 0$ for all $m\geq 1$. Therefore, from Eq. (\ref{eq20141111}) and Eq. (\ref{eqabc}) we have
\begin{equation*}
\mathrm{Re}\Big[2i\beta  z_2 {Q_0}_{z_2}(z_2)+i a_1(z_2)\Big(1+Q_0^2(z_2)\Big)\Big]\equiv 0
\end{equation*}
on $\Delta_{\epsilon_0}$. Hence, the solution $Q_0(z_2)=\tan (R(z_2))$ for all $z_2\in \Delta_{\epsilon_0}$, where $R$ is given in the claim  (see Lemma \ref{pde1} in Appendix A.2), and hence $F(z_2,t)=Q_0(z_2)t=\tan (R(z_2))t$ for all $(z_2,t)\in \Delta_{\epsilon_0}\times (-\delta_0,\delta_0)$, as claimed.

\smallskip
\noindent
{\bf Case B.~}{\boldmath $Q_1(0)\ne 0$.} In this case, it follows from (\ref{etn0}) that
\begin{equation*}
\begin{split}
&\mathrm{Re}\Big\{ \Big(\frac{1}{2}+\frac{F_t(z_2,t)}{2i}\Big) \frac{1}{2 Q_1(0)} \Big[\exp \Big(2Q_1(0)\big(it-P(z_2)-F(z_2,t) \big) \Big)-1\Big] a_1(z_2)\\
&\quad +\Big(P'(z_2)+{F}_{z_2}(z_2,t)\Big) i\beta z_2\Big\}=0,
\end{split}
\end{equation*}
or equivalently 
\begin{equation}\label{eq201411}
\begin{split}
&\mathrm{Re}\Big\{ i\beta z_2 P'(z_2)+\frac{\exp\big(-2Q_1(0)P(z_2)\big)-1}{2Q_1(0)} \Big(\frac{1}{2}+\frac{Q_0(z_2)}{2i}\Big)a_1(z_2)\Big\}\\
&+e^{-2Q_1(0)P(z_2)}\mathrm{Re}\Big\{ \Big[\frac{i+F_t(z_2,t)}{2iQ_1(0)} \exp \Big(2Q_1(0)\big(it-F(z_2,t) \big) \Big)-\frac{i+F_t(z_2,0)}{2iQ_1(0)}\Big ] a_1(z_2)\Big\}\\
&+\mathrm{Re}\Big\{ i\beta z_2{F}_{z_2}(z_2,t)-\frac{F_t(z_2,t)-\tan(R(z_2))}{2iQ_1(0)}a_1(z_2)\Big\}=0
\end{split}
\end{equation}
for all $z_2\in \mathbb C$ and for all $t\in\mathbb R$ with $|z_0|<\epsilon_0$ and $|t|<\delta_0$. 

Now we shall show the following assertions:
\begin{itemize}
\item[(i)] $\mathrm{Re}\Big\{\Big[\big(i+F_t(z_2,t)\big)\exp\Big(2Q_1(0)\big(it-F(z_2,t)\big)\Big)-\big(i+F_t(z_2,0)\big)\Big]ia_1(z_2)\Big\}=0$ ;
 \item[(ii)] $\mathrm{Re}\Big[4iQ_1(0)\beta  z_2 F_{z_2}(z_2,t)+\Big(F_t(z_2,t)-\tan(R(z_2))\Big) ia_1(z_2)\Big]= 0$;
 \item[(iii)]$\mathrm{Re}\Big(i\beta z_2 P'(z_2)\Big)=-\frac{\exp\big(-2Q_1(0)P(z_2)\big)-1}{2Q_1(0)} \mathrm{Re}\Big[\Big(\frac{1}{2}+\frac{Q_0(z_2)}{2i}\Big)a_1(z_2)\Big] $
\end{itemize}
for all $(z_2,t)\in \Delta_{\epsilon_0}\times (-\delta_0,\delta_0)$.

Indeed, inserting $t=0$ into (\ref{eq201411}) one has $\mathrm{(iii)}$. Since the function $F(z_2,t)=\sum_{n=1}^\infty Q_{n-1}(z_2) t^n$ is real-analytic in a neighborhood of $0\in \mathbb C\times \mathbb R$, $P(z_2)$ vanishes to infinite order at $z_2=0$, and $a_1$ is holomorphic, it follows the assertion $\mathrm{(i)}$. Finally, $\mathrm{(ii)}$ is easily obtained.

By $\mathrm{(i)}$, it follows from Lemma \ref{pde3} in Appendix A.2 with $\alpha =2Q_1(0)$ that 
\[
 F(z_2,t)=\begin{cases}
 -\frac{1}{2Q_1(0)}\log \Big|\frac{\cos \big(R(z_2)+2Q_1(0)t\big)}{\cos (R(z_2))} \Big| &~\text{if}~ Q_1(0)\ne 0\\
 \tan(R(z_2))t  &~\text{if}~ Q_1(0) = 0
\end{cases} 
 \]
for all $(z_2,t)\in \Delta_{\epsilon_0}\times (-\delta_0,\delta_0)$. We note that 
$$
2\mathrm{Re}\Big(i\beta z_2 R_{z_2}(z_2)\Big)=- \mathrm{Re}\big(ia_1(z_2)\big)
$$
 for all $z_2\in \Delta_{\epsilon_0}$. Hence, by Corollary \ref{pde4}  in Appendix A.2 we conclude that  Eq. $\mathrm{(ii)}$  automatically holds. Finally, by Eq. $\mathrm{(iii)}$ and Lemma \ref{pde1} in Appendix A.2 with $\alpha=2Q_1(0)$, we conclude that the function $P(z_2)$ has the form as in the claim.

Altogether, the claim is proved.
\end{proof}

In conclusion, Claims \ref{c1}, \ref{c2},\ldots, and \ref{c6} complete the proof of Theorem \ref{T1}, in which $a(z_2):=a_1(z_2)/\beta$ and $\alpha:=2Q_1(0)$, (modulo Lemmas \ref{pde1}, \ref{pde2}, and \ref{pde3}, and Corollary \ref{pde4} which we prove in Appendix A.2). \hfill $\Box\;$
\section{Functions vanishing to infinite order}\label{S5}
In this section, we will introduce the condition $(\mathrm{I})$ and give several examples of functions defined on the open unit disc in the complex plane with infinite order of vanishing at the origin.

\begin{define} \label{def1} We say that a real $\mathcal{C}^1$-smooth function $f$ defined on a neighborhood $U$ of the origin in $\mathbb C$  satisfies the \emph{condition $(I)$} if 
\begin{itemize}
\item[(I.1)]$ \limsup_{\tilde U \ni z\to 0} |\mathrm{Re}(b z^k \frac{f'(z)}{f(z)})| =+\infty$;
\item[(I.2)] $ \limsup_{\tilde U \ni z\to 0} | \frac{f'(z)}{f(z)}| =+\infty$
\end{itemize}
for all $k=1,2,\ldots$ and for all $b\in \mathbb C^*$, where $\tilde U :=\{z\in U: f(z)\ne 0\}$.
\end{define}
\begin{example}
The function $P(z)=e^{-C/|\mathrm{Re}(z)|^\alpha}$ if $\mathrm{Re}(z)\ne 0$ and $P(z)=0$ if otherwise, where $C, \alpha>0$, satisfies the condition $(\mathrm{I})$. Indeed, a direct computation shows that 
$$
P'(z)=P(z)\frac{C\alpha}{2|\mathrm{Re}(z)|^{\alpha+1}}
$$
for all $z\in \mathbb C$ with $\mathrm{Re}(z)\ne 0$. Therefore, it is easy to see that $ | P'(z)/P(z)| \to+\infty$ as $z\to 0$ in the domain $\{z\in \mathbb C\colon \mathrm{Re}(z)\ne 0\}$. 

Now we shall prove that the condition $(\mathrm{I.1})$ holds. Let $k$ be an arbitrary positive integer. Let $z_l:=1/l+i/ l^\beta$, where $0<\beta<\min\{1,\alpha/(k-1)\}$ if $k>1$ and $\beta=1/2$ if $k=1$, for all $l\in \mathbb N^*$. Then $z_l \to 0$ as $l\to\infty$ and $\mathrm{Re}(z_l)=1/l\ne 0$ for all $l\in \mathbb N^*$. Moreover, for each $b\in \mathbb C^*$ we have that
\begin{equation*}
\begin{split}
|\mathrm{Re} \Big(b z_l^{k} \frac{P'(z_l)}{P(z_l)}\Big)|&\gtrsim  \frac{l^{\alpha+1}}{l^{\beta (k-1)+1}}                                                              =l^{\alpha-\beta (k-1)}.
\end{split}
\end{equation*}
This implies that 
$$
\lim_{l\to\infty}|\mathrm{Re} \Big(b z_l^{k} \frac{P'(z_l)}{P(z_l)}\Big)|=+\infty .
$$
Hence, the function $P$ satisfies the condition $(\mathrm{I})$.
\end{example}
\begin{remark} \label{R2}

\noindent
i) Any rotational function $P$ does not satisfy the condition $(\mathrm{I.1})$ because $\mathrm{Re}(iz P'(z))=0$ (see \cite{Kim-Ninh} or \cite{By1}).

\noindent
ii) It follows from \cite[Lemma 2]{Kim-Ninh} that if $P$ is a non-zero $\mathcal{C}^1$-smooth function defined on a neighborhood $U$ of the origin in $\mathbb C$, $P(0)=0$, and $\tilde U:=\{z\in U\colon P(z)\ne 0\}$ contains a $\mathcal{C}^1$-smooth curve $\gamma : (0,1]\to \tilde U$ such that $\gamma'$ stays bounded on $(0,1]$ and $\lim_{t\to 0^-} \gamma(t)=0$, then $P$ satisfies the condition $(\mathrm{I.2})$.  

\end{remark}
\begin{lemma} Suppose that $g:(0,1] \to \mathbb R$ is a $\mathcal{C}^1$-smooth  unbounded function. Then we have $\limsup_{t\to 0^+} t^{\alpha}|g'(t)|=+\infty$ for any real number $\alpha <1$.
\end{lemma}
\begin{proof} Fix an arbitrary $\alpha<1$. Suppose that, on the contrary, $\limsup_{t\to 0^+} t^\alpha|g'(t)|<+\infty$. Then there is a constant $C>0$ such that 
$$
|g'(t)|\leq \frac{C}{t^\alpha}, \; \forall \; 0<t<1.
$$
We now have the following estimate
\begin{equation*}
\begin{split}
|g(t)|&\leq |g(1)|+\int_t^1|g'(\tau)| d\tau\leq  |g(1)|+C\int_t^1\frac{d\tau}{\tau^\alpha}\\
           &\leq  |g(1)|+ \frac{C}{1-\alpha} (1-t^{1-\alpha})\lesssim 1.
\end{split}
\end{equation*}
However, this is impossible since $g$ is unbounded on $(0,1]$, and thus the lemma is proved.
\end{proof}

In general, the above lemma does not hold for $\alpha\geq 1$. This follows from that $|t^{1+\beta}\frac{d}{dt}\frac{1}{t^\beta}|= \beta$ and $|t\frac{d}{dt}\log (t)|= 1$ for all $0<t<1$, where $\beta>0$. However, the following lemmas show that there exists such a function $g$ such that $\liminf_{t\to 0^+}\sqrt{t}|g'(t)|<+\infty$ and $\limsup_{t\to 0^+} t^\beta |g'(t)|=+\infty$ for all $\beta<2$. Furthermore, several examples of smooth functions vanishing to infinite order at the origin in $\mathbb C$ and satisfying the condition $(\mathrm{I})$ are constructed.  
\begin{lemma}\label{lemma1} There exists a $\mathcal{C}^\infty$-smooth real-valued
function $g: (0,1)\to \mathbb R$ satisfying
\begin{enumerate}
\item[(\romannumeral1)] $g(t)\equiv -2n$ on the closed interval
$ \Big[\dfrac{1}{n+1}\Big(1+\dfrac{1}{3n}\Big),\dfrac{1}{n+1}\Big(1+\dfrac{2}{3n}\Big) \Big]$ for $n=4,5,\ldots$;
\item[(\romannumeral2)] $g(t)\approx\dfrac{-1}{t}$, $\forall ~t\in (0,1)$;
\item[(\romannumeral3)] for each $k\in \mathbb N$ there exists $C(k)>0$, depending only on $k$, such that $|g^{(k)}(t)|\leq \dfrac{C(k)}{t^{3k+1}},\;\forall \; t\in (0,1)$.
\end{enumerate}
\end{lemma}
\begin{remark}
Let
\begin{equation*}
P(z):=
\begin{cases}
\exp(g(|z|^2)) & \text{if}~0<|z|<1 \\
0  &\text{if}~z=0.
\end{cases}
\end{equation*}
Then this function is a $\mathcal{C}^\infty $-smooth function on the open unit disc $ \Delta $ that
vanishes to infinite order at the origin. Moreover, we see that $P'( \frac{2n+1}{2n(n+1)})=0$ for any $n\geq 4$, and hence $\liminf_{z\to 0}|P'(z)|/P(z)=0$.

Lemma \ref{lemma1} was stated in \cite{Kim-Ninh} without proof. A detailed proof of this lemma is given in Appendix A.1. 
\end{remark}
\begin{lemma}
Let $h:  (0,+\infty)\to \mathbb R$ be the piecewise linear function such that $h(a_n)=h(b_n)=2^{2\cdot 4^{n-1}}$, $h(1/2)=\sqrt{2}$ and $h(t)=0$ if $t\geq 1$, where $a_n=1/2^{4^n},~a_0=1/2,~b_n=(a_n+a_{n-1})/2$ for every $ n\in \mathbb N^*$. Then the function $f: (0,1) \to \mathbb R$ given by
$$
f(t)=-\int_t^1 h(\tau)d\tau
$$
satisfies:
\begin{itemize}
\item[(i)] $f'(a_n)=\frac{1}{\sqrt{a_n}}$ for every $n\in \mathbb N^*$;
\item[(ii)] $f'(b_n)\sim \frac{1}{4 b^2_n}$ as $n\to\infty$;
\item[(iii)] $-\frac{1}{t}\lesssim f(t)\lesssim -\frac{1}{t^{1/16}}$, $\forall~ 0<t<1$.
\end{itemize}
\end{lemma}
\begin{proof}
We have $f'(a_n)=h(a_n)=2^{2\cdot 4^{n-1}}=\frac{1}{\sqrt{a_n}}$, which proves $\mathrm{(i)}$. Since $b_{n}=(a_n+a_{n-1})/2\sim a_{n-1}/2$ as $n\to\infty$, we have $f'(b_n)=h(b_n)=2^{2\cdot 4^{n-1}}=\frac{1}{a^2_{n-1}}\sim \frac{1}{4 b^2_n} $ as $n\to \infty$. So, the assertion $\mathrm{(ii)}$ follows. Now we shall show $\mathrm{(iii)}$. For an abitrary real number $t\in (0,1/16)$, denote by $N$ the positive integer such that 
$$
1/2^{4^{N+1}}\leq t<1/2^{4^{N}}.
$$ 
Then it is easy to show that  
\begin{equation*}
\begin{split}
f(t)&\leq -\int_{a_{N}}^{b_{N}}h(\tau) d\tau =-\frac{1}{2} 2^{2\cdot 4^{N-1}}(1/2^{4^{N-1}}-1/2^{4^{N}})\\
 &\leq  -\frac{1}{2} 2^{4^{N-1}}+\frac{1}{8}\leq  -\frac{1}{2} \frac{1}{t^{1/16}}+\frac{1}{8}\lesssim -\frac{1}{t^{1/16}};\\
f(t)&\geq -2\int_{a_{N+1}}^{b_{N+1}}h(\tau) d\tau- \int_{a_N}^1h(\tau)d\tau\\
     &\geq -2 h(a_{N+1})(b_{N+1}-a_{N+1})- h(a_N)(1-a_N)\\
     &\geq -2^{2\cdot 4^{N}}(1/2^{4^{N}}-1/2^{4^{N+1}})-2^{2\cdot 4^{N-1}}(1-1/2^{4^{N}})\\
      &\gtrsim -\frac{1}{t}
\end{split}
\end{equation*}
for any $0<t<1/16$. Thus $\mathrm{(iii)}$ is shown.
\end{proof}
\begin{remark}

\noindent
$\mathrm{i)}$ We note that $f$ is $\mathcal{C}^1$ -smooth, increasing, and concave on the interval $(0,1)$. By taking a  suitable regularization of the function $f$ as in the proof of Lemma \ref{lemma1}, we may assume that it is $\mathcal{C}^\infty$-smooth and still satisfies the above properties $\mathrm{(i)}, \mathrm{(ii)}$, and $\mathrm{(iii)}$. In addition, for each $k\in \mathbb N $ there exist $C(k)>0$ and $d(k)>0$, depending only on $k$, such that $|f^{(k)}(t)|\leq \dfrac{C(k)}{t^{d(k)}},\forall\; t\in (0,1)$. Thus the function $R(z)$ defined by
\begin{equation*}
R(z):=
\begin{cases}
\exp(f(|z|^2)) & \text{if}\; 0<|z|<1 \\
0  &\text{if}\; z=0
\end{cases}
\end{equation*}
is $\mathcal{C}^\infty$-smooth and vanishes to infinite order at the origin. Moreover, we have 
$\liminf_{z\to 0}|R'(z)/R(z)| <+\infty$ and $\limsup_{z\to 0}|R'(z)/R(z)| =+\infty $.

\noindent
$\mathrm{ii)}$ Since the functions $P,R$ are rotational, they do not satisfy the condition $(\mathrm{I})$ (cf. Remark \ref{R2}). On the other hand, the functions $\tilde P(z):=P(\mathrm{Re}(z) )$ and $\tilde R(z):= R(\mathrm{Re}(z) )$ satisfy the condition $(\mathrm{I})$. Indeed, a simple calculation shows 
$$
\tilde R'(z)=\tilde R(z)  f'(|\mathrm{Re}(z)|^2) \mathrm{Re}(z)
$$
for any $z\in \mathbb C$ with $|\mathrm{Re}(z)|<1$. By the above property $\mathrm{(ii)}$, it follows that $\limsup_{z\to 0 } |\tilde R'(z)|/\tilde R(z)=+\infty$. Moreover, for each $k\in \mathbb N^*$ and each $b\in \mathbb C^*$ if we choose a sequence $\{z_n\}$ with $z_n:=\sqrt{b_n}+i(\sqrt{b_n})^\beta$, where $0<\beta<\min\{1,2/(k-1)\}$ if $k>1$ and $\beta=1/2$ if $k=1$, then $z_n \to 0$ as $n\to \infty$ and 
$$ 
|\mathrm{Re}\Big( b z_n^k \frac{\tilde R'(z_n)}{\tilde R(z_n)}\Big)| \gtrsim \frac{(\sqrt{b_n})^{(k-1)\beta+2}}{b_n^2} \to +\infty
$$
as $n\to \infty$. Hence, $\tilde R$ satisfies the condition $(\mathrm{I})$. Now it follows from the construction of the function $g$  in the proof of Lemma \ref{lemma1} (cf. Appendix A.1) that $g'(\frac{1}{n})\sim 3n^2 $ as $n\to \infty$. Therefore, using the same argument as above we conclude that $\tilde P$ also satisfies the condition $(\mathrm{I})$.

It is not hard to show that the above functions such as $P,R, \tilde P,\tilde R$ are not subharmonic. Up to now it is unknown that there exists a $\mathcal{C}^\infty$-smooth subharmonic function $P$  defined on the unit disc such that $\nu_0(P)=+\infty$ and $\liminf_{z\to 0}|P'(z)/P(z)| <+\infty$.
\end{remark}
\section{Proof of Theorem \ref{T3}}\label{S6}

This section is entirely devoted to the proof of Theorem \ref{T3}. Let $M=\{(z_1,z_2) \in \mathbb C^2 \colon \mathrm{Re}~z_1 + P(z_2) + (\mathrm{Im}~z_1)Q(z_2,\mathrm{Im}~z_1)=0\}$ be the real hypersurface germ at $0$ described in the hypothesis of Theorem \ref{T3}. Our present goal is to show that there is no non-trivial holomorphic vector field vanishing at the origin and tangent to $M$.
\smallskip

For the sake of smooth exposition, we shall present the proof in two subsections. In Subsection \ref{S6.1}, several technical lemmas are introduced. Then the proof of Theorem \ref{T3} is presented in Subsection \ref{S6.2}. Throughout what follows, for $r>0$ denote by $\tilde\Delta_r:=\{z_2\in \Delta_r\colon P(z_2)\ne 0 \}$.
\subsection{Technical lemmas} \label{S6.1}

Since $P$ satisfies the condition $(\mathrm{I})$, it is not hard to show the following two lemmas. 
\begin{lemma}\label{l3}  Let $P$ be a function defined on $\Delta_{\epsilon_0}$ ($\epsilon_0>0$) satisfying the condition $(\mathrm{I})$. If $a, b$ are complex numbers and if $g_0, g_1, g_2$ are $\mathcal{C}^\infty$-smooth functions defined on $\Delta_{\epsilon_0}$ satisfying:
\begin{itemize}
\item[(i)] $g_0(z)=O(|z|)$, $g_1(z) = O(|z|^{\ell+1})$, $g_2(z) = o(|z|^m)$, and
\item[(ii)] $\mathrm{Re} \Big[a z^m+\frac{b}{P^n(z)}\Big( z^{\ell+1}\big(1+g_0(z)\big) \frac{P'(z)}{P(z)}
+g_1(z) \Big)\Big]= g_2(z)$ 
\end{itemize}
for every $z \in\tilde \Delta_{\epsilon_0}$ and for any non-negative integers $\ell, m$, except the case that $m=0$ and $\mathrm{Re}(a)=0$, then $a=b=0$.
\end{lemma}
\begin{proof}
The proof follows easily from the condition $(\mathrm{I.1})$.
\end{proof}
\begin{lemma}\label{l8} Let $P$ be a function defined on $\Delta_{\epsilon_0}$ ($\epsilon_0>0$) satisfying the condition $(\mathrm{I})$. Let $B\in \mathbb C^*$ and $m\in \mathbb N^*$. Then there exists $\alpha \in \mathbb R$ small enough such that
\begin{equation*}
\limsup_{\tilde \Delta_{\epsilon_0}\ni z\to 0}|\mathrm{Re} \Big(B (i\alpha-1)^m  P'(z)/P(z)\Big)|=+\infty.
\end{equation*}
\end{lemma}
\begin{proof}
Since $P$ satisfies the condition $(\mathrm{I.2})$, there exists a sequence $\{z_k\}\subset \tilde \Delta_{\epsilon_0}$ converging to $0$ such that $\lim_{k\to \infty}  P'(z_k)/P(z_k)=\infty$.  We can write 
\begin{equation*}
\begin{split}
B P'(z_k)/P(z_k)&=a_k+i b_k,\quad k=1,2,\ldots; \\
(i\alpha-1)^m&=a(\alpha)+i b(\alpha).
\end{split}
\end{equation*}

We note that $|a_k|+|b_k|\to +\infty$ as $k\to \infty$. Therefore, passing to a  subsequence if necessary, we only consider two following cases.
\smallskip

\noindent
{\bf Case 1.} {\boldmath $\lim_{k\to \infty} a_k=\infty$ and $|\frac{b_k}{a_k}|\lesssim 1$ .} Since $a(\alpha)\to (-1)^m$ and $b(\alpha)\to 0$ as $\alpha \to 0$, if $\alpha$ is small enough then
\begin{equation*}
\begin{split}
\mathrm{Re} \Big(B (i\alpha-1)^m  P'(z_k)/P(z_k)\Big)&=a(\alpha) a_k - b(\alpha) b_k \\
                                                                                      &=a_k\Big(a(\alpha)-b(\alpha) \frac{b_k}{a_k}\Big) \to \infty
\end{split}
\end{equation*}
as $k\to \infty$.

\noindent
{\bf Case 2.} {\boldmath $\lim_{k\to \infty} b_k=\infty$ and $\lim_{k\to \infty} |\frac{a_k}{b_k}|=0$ .} Fix a real number $\alpha$ such that $b(\alpha)\ne 0$. Then we have
\begin{equation*}
\begin{split}
\mathrm{Re} \Big(B (i\alpha-1)^m  P'(z_k)/P(z_k)\Big)&=a(\alpha) a_k - b(\alpha) b_k \\
                                                                                      &=b_k\Big(a(\alpha)\frac{a_k}{b_k}-b(\alpha)\Big) \to \infty
\end{split}
\end{equation*}
as $k\to \infty$. Hence, the proof is complete.
\end{proof}
\subsection{Proof of Theorem \ref{T3}}\label{S6.2}
The CR hypersurface germ $(M,0)$ at the origin in $\CC^2$ under consideration is defined by the equation $\rho(z_1, z_2) = 0$, where
$$
\rho (z_1, z_2) = \mathrm{Re}~z_1 + P(z_2) + (\mathrm{Im}~z_1)\ Q(z_2, \mathrm{Im}~z_1) = 0,
$$
where $P, Q$ are $\mathcal{C}^\infty$-smooth functions satisfying the three conditions specified in the hypothesis of Theorem \ref{T3}, stated in Section \ref{S2}.  Recall that $P$ vanishes to infinite order at $z_2=0$ in particular.

Then we consider a holomorphic vector field $H=h_1(z_1,z_2)\frac{\partial}{\partial z_1}+h_2(z_1,z_2)\frac{\partial}{\partial z_2}$ defined on a neighborhood of the origin. We only consider $H$ that is tangent to $M$, which means that they satisfy the identity
\begin{equation}\label{eq221}
(\mathrm{Re}~ H) \rho(z)=0,\; \forall z \in M.
\end{equation}

The goal is to show that $H\equiv 0$. Indeed, striving for a contradiction, suppose that $H\not\equiv 0$. We notice that if $h_2\equiv 0$ then (\ref{eq221}) shows that $h_1\equiv 0$. Thus, $h_2\not\equiv 0$. 

Now we are going to prove that $h_1\equiv 0$. Indeed, suppose that $h_1\not\equiv 0$. Then we can expand $h_1$ and $h_2$ into the Taylor series at the origin so that
$$
h_1(z_1,z_2)=\sum\limits_{j,k=0}^\infty a_{jk} z_1^j z_2^k
\text{ and }
h_2(z_1,z_2)=\sum\limits_{j,k=0}^\infty b_{jk} z_1^jz_2^k,
$$
where $a_{jk}, b_{jk}\in \mathbb C$. We note that $a_{00}=b_{00}=0$ since $h_1(0,0)=h_2(0,0)=0$.

 Next, let us denote by $j_0$ the smallest integer such that $a_{j_0 k}\ne 0$ for some integer $k$. Then let $k_0$ be the smallest integer such that $a_{j_0 k_0}\ne 0$. Similarly, let $m_0$ be the smallest integer such that $b_{m_0 n}\ne 0$ for some integer $n$. Then denote by $n_0$ the smallest integer  such that $b_{m_0 n_0}\ne 0$. One remarks that $j_0\geq 1$ if $k_0=0$ and $m_0\geq 1$ if $n_0=0$.

 Following the arguments in the proof of Theorem \ref{T1}, one obtains that
\begin{equation}\label{eq28}
\begin{split}
&\mathrm{Re} \Big[\frac{1}{2} a_{j_0k_0}(i\alpha -1)^{j_0}(P(z_2))^{j_0}z_2^{k_0}+ b_{m_0n_0}(i\alpha -1)^{m_0}(z_2^{n_0}+o(|z_2|^{n_0})(P(z_2))^{m_0} \\
&\quad \times\Big(P'(z_2)+\alpha P(z_2)Q_{z_2}(z_2, \alpha P(z_2))\Big)   \Big ]=o(P(z_2)^{j_0}|z_2|^{k_0})
\end{split}
 \end{equation}
for all $|z_2|<\epsilon_0$ and for any $\alpha \in \mathbb R$. We note that in the case $k_0=0$ and $\mathrm{Re}(a_{j_0 0})=0$, $\alpha$ can be chosen in such a way that $\mathrm{Re}\big( (i\alpha-1)^{j_0}a_{j_0 0}\big)\ne 0$. Then the above equation yields that $j_0>m_0$.

We now divide the argument into two cases as follows.
\smallskip

\noindent
{\bf Case 1.} {\boldmath $n_0\geq 1$.} In this case $(\ref{eq28})$ contradicts Lemma \ref{l3}.

\medskip

\noindent
{\bf Case 2.} {\boldmath $n_0=0$.} Since $P$ satisfies the condition $(\mathrm{I})$ and $m_0\geq 1$, by Lemma \ref{l8} we can choose a real number $\alpha$ such that 
$$
\limsup_{\tilde\Delta_{\epsilon_0} \ni z_2\to 0}| \mathrm{Re}\Big(b_{m0}(i\alpha-1)^m P'(z_2)/P(z_2)\Big)|=+\infty,
$$
where $\epsilon_0>0$ is small enough. Therefore, $(\ref{eq28})$ is a contradiction, and thus $h_1\equiv 0$ on a neighborhood of $(0,0)$ in $\mathbb C^2$.

Since $h_1\equiv 0$, it follows from (\ref{eq201311}) with $t=0$ that
$$
\mathrm{Re}~\Big [ \sum_{m,n=0}^\infty  b_{mn} z^n_2P'(z_2)\Big ]=0
$$
for every $z_2$ satisfying $|z_2|<\epsilon_0$, for some $\epsilon_0>0$ sufficiently small. Since $P$ satisfies the condition $(\mathrm{I.1})$, we conclude that $b_{mn}=0$ for every
$m \geq 0,n\ge 1$. We now show that $b_{m0}=0$ for every $m\in \mathbb N^*$. 
Indeed, suppose otherwise. Then let $m_0$ be the smallest positive integer such that $b_{m_00}\ne 0$.  It follows from (\ref{eq27}) in the proof of Theorem \ref{T1} that 
$$
\mathrm{Re}~ \Big(b_{m_0 0}(i\alpha-1)^{m_0} P'(z_2)/P(z_2) \Big)
$$
is bounded on $\tilde \Delta_{\epsilon_0}$ with $\epsilon_0>0$ small enough for any $\alpha \in \mathbb R$ small enough. By Lemma \ref{l8}, this is again impossible.

Altogether, the proof of Theorem \ref{T3} is complete. \hfill $\Box\;$

\section*{Appendix A}
\subsection*{A.1. Proof of Lemma \ref{lemma1}}
Let $G:  (0,+\infty)\to \mathbb R$ be the piecewise linear function such that $G(a_n-\epsilon_n)=G(b_n+\epsilon_n)=-2n$ and $G(x)=-8$ if $x\geq\frac{9}{40}$, where $a_n=\frac{1}{n+1}(1+\frac{1}{3n}),~ b_n=\frac{1}{n+1}(1+\frac{2}{3n})$, and $\epsilon_n=\frac{1}{n^3} $ for every $n\geq 4$.

Let $\psi$ be a $\mathcal{C}^\infty$-smooth function on $\mathbb R$ given by
\begin{equation*}
\psi(x)= C
\begin{cases}
e^{-\frac{1}{1-|x|^2}}  & \text{if} ~|x|<1\\
0 &\text{if} ~|x|\geq 1,
\end{cases}
\end{equation*}
where $C>0$ is chosen so that $\int_{\mathbb R}\psi(x) dx=1$. For $\epsilon >0$, set $\psi_\epsilon:= \frac{1}{\epsilon}\psi(\frac{x}{\epsilon})$. For $n\geq 4$, let $g_n$ be the $\mathcal{C}^\infty$-smooth on $\mathbb R$ defined by the following convolution
$$
g_n(x):= G\ast \psi_{\epsilon_{n+1}}(x)=\int_{-\infty}^{+\infty} G(y)\psi_{\epsilon_{n+1}}(y-x)dy.
$$
Now we show the following.
\begin{itemize}
\item[(a)] $g_n(x)=G(x)=-2n$ if $a_n\leq x\leq b_n$;\\
\item[(b)] $g_n(x)=G(x)=-2(n+1)$ if $ a_{n+1}\leq x\leq b_{n+1}$;\\
\item[(c)] $|g^{(k)}_n(x)|\leq \frac{2(n+1)\|\psi^{(k)}\|_1}{\epsilon_{n+1}^k}$  if $a_{n+1}\leq x\leq b_n$.
\end{itemize}
Indeed, for $a_{n+1}\leq x\leq b_n$ we have 
\begin{equation*}
\begin{split}
g_n(x)&=\int_{-\infty}^{+\infty} G(y)\psi_{\epsilon_{n+1}}(y-x)dy\\
           &=\frac{1}{\epsilon_{n+1}}\int_{-\infty}^{+\infty} G(y)\psi(\frac{y-x}{\epsilon_{n+1}})dy\\
            &=\int_{-1}^{+1} G(x+t\epsilon_{n+1})\psi(t)dt,
\end{split}
\end{equation*}
where we use a change of variable $t=\dfrac{y-x}{\epsilon_{n+1}}$.

If $a_n\leq x\leq b_n$, then $a_n-\epsilon_n<a_n-\epsilon_{n+1}\leq x+t\epsilon_{n+1} \leq b_n+\epsilon_{n+1}<b_n+\epsilon_n$ for all $-1\leq t \leq 1 $. Therefore, 
$$
g_n(x)=\int_{-1}^{+1} G(x+t\epsilon_{n+1})\psi(t)dt=-2n\int_{-1}^{+1} \psi(t)dt=-2n,
$$
which proves $\mathrm{(a)}$. Similarly, if  $a_{n+1}\leq x\leq b_{n+1}$, then $a_{n+1}-\epsilon_{n+1}\leq x+t\epsilon_{n+1} \leq b_{n+1}+\epsilon_{n+1}$ for every $-1\leq t \leq 1 $. Hence, 
$$
g_n(x)=\int_{-1}^{+1} G(x+t\epsilon_{n+1})\psi(t)dt=-2(n+1)\int_{-1}^{+1} \psi(t)dt=-2(n+1),
$$
which finishes $\mathrm{(b)}$. Moreover, we have the following estimate
\begin{equation*}
\begin{split}
|g^{(k)}_n(x)|&=\frac{1}{\epsilon^{k+1}_{n+1}}|\int_{-\infty}^{+\infty} G(y)\psi^{(k)}(\frac{y-x}{\epsilon_{n+1}})dy|\\
            &=\frac{1}{\epsilon^{k}_{n+1}}|\int_{-1}^{+1} G(x+t\epsilon_{n+1})\psi^{(k)}(t)dt|\\
           &\leq\frac{1}{\epsilon^{k}_{n+1}}\int_{-1}^{+1}| G(x+t\epsilon_{n+1})||\psi^{(k)}(t)|dt\\
&\leq\frac{2(n+1)}{\epsilon^{k}_{n+1}}\int_{-1}^{+1} |\psi^{(k)}(t)|dt\\
&=\frac{2(n+1)\|\psi^{(k)}\|_1}{\epsilon^{k}_{n+1}}
\end{split}
\end{equation*}
 for $a_{n+1}\leq x\leq b_n$, where we use again a change of variable $t=\dfrac{x-y}{\epsilon_{n+1}}$ and the last inequality in the previous equation follows from the fact that $|G(y)|\leq 2(n+1)$ for all $a_{n+1}-\epsilon_{n+1}\leq y\leq b_n+\epsilon_n$. So, the assertion $\mathrm{(c)}$ is shown.

Now because of properties $\mathrm{(a)}$ and $\mathrm{(b)}$ the function
\begin{equation*}
g(x)= 
\begin{cases}
-8  & \text{if}~ x\geq\frac{9}{40}\\
g_n(x)  & \text{if}~ a_{n+1}\leq x\leq b_n,\;  n=4,5,\ldots,
\end{cases}
\end{equation*}
is well-defined. From the property $\mathrm{(c)}$, it is easy to show that $|g^{(k)}(x)|\lesssim \frac{1}{x^{3k+1}}$ for $k=0,1,\ldots$ and for every $x\in (0,1)$, where the constant depends only on $k$. Thus this proves $\mathrm{(iii)}$, and the assertions $\mathrm{(i)}$ and $\mathrm{(ii)}$ are obvious. Hence, the proof is complete.
\hfill $\Box\;$
\subsection*{A.2. Several  differential equations}
In this subsection, we are going to prove several lemmas and a corollary used in the proof of Theorem \ref{T1}.
\begin{lemma}\label{pde1}
Let $a_1(z_2)=\beta\sum_{n=1}^\infty a_n z_2^n$ be a non-zero holomorphic function on $\Delta_{\epsilon_0}~(\beta\in \mathbb R^*,\epsilon_0>0, a_n\in \mathbb C~\text{for all}~n\in \mathbb N^*)$. Let $Q_0,P_1,P$ be $\mathcal{C}^1$-smooth functions on $\Delta_{\epsilon_0}$ with $P_1,P$ are positive on $\Delta_{\epsilon_0}^*$ satisfying the following differential equations:
\begin{itemize}
\item[(i)] $\mathrm{Re}\Big[2i\beta  z_2 {Q_0}_{z_2}(z_2)+i a_1(z_2)\Big(1+Q_0^2(z_2)\Big)\Big]\equiv 0$;
\item[(ii)] $\mathrm{Re}\Big[2i\beta z_2 {P_1}_{z_2}(z_2)-\Big(1+\frac{Q_0(z_2)}{i}\Big)a_1(z_2)P_1(z_2)\Big]\equiv 0$;
\item[(iii)]$\mathrm{Re}\Big[2i\beta z_2 {P}_{z_2}(z_2)+\frac{\exp\big(-\alpha P(z_2)\big)-1}{\alpha} \Big(1+\frac{Q_0(z_2)}{i}\Big)a_1(z_2)\Big]\equiv 0 $
\end{itemize}
on $\Delta_{\epsilon_0}$, where $\alpha\in \mathbb R^*$. Then we have
\begin{equation*}
\begin{split}
 Q_0(z_2) &=\tan \Big[q(|z_2|)- \mathrm{Re}\Big(\sum_{n=1}^\infty \frac{a_n}{n} z_2^n\Big)\Big];\\
P_1(z_2)&=\exp\Big[ p(|z_2|)+\mathrm{Re}\Big(\sum_{n=1}^\infty  \frac{a_n}{in}z_2^n\Big )-\log \Big|\cos\Big(q(|z_2|)- \mathrm{Re}\big(\sum_{n=1}^\infty\frac{a_n}{n} z_2^n\big)\Big )\Big|\Big];\\
 P(z_2)&=\frac{1}{\alpha}\log\Big[1+\alpha P_1(z_2)\Big]
\end{split}
\end{equation*}
for all $z_2\in \Delta_{\epsilon_0}^*$, where $q,p$ are $\mathcal{C}^1$-smooth in  $(0,\epsilon_0)$ and are chosen so that $Q_0,P_1, P$ are $\mathcal{C}^1$-smooth on $\Delta_{\epsilon_0}$.
\end{lemma}
\begin{proof}
We first find solutions of the differential equation $\mathrm{(i)}$. Indeed, it follows from $\mathrm{(i)}$ that 
$$
\frac{2\mathrm{Re}\big(i\beta z_2 {Q_0}_{z_2}(z_2)\big)}{1+Q_0^2(z_2)}=-\mathrm{Re}\big( ia_1(z_2)\big)=-\beta \; \mathrm{Re}\big( i\sum_{n=1}^\infty a_n z_2^n\big)
$$
for all $z_2\in \Delta_{\epsilon_0}$. 
For a fixed positive number $0<r<\epsilon_0$, set $u(t):=Q_0(r e^{it})$ for every $t\in \mathbb R$. Then one has $u'(t)=2 \mathrm{Re}( i r e^{it} {Q_0}_{z_2}(r e^{it}))$, and hence
$$
 \frac{u'(t)}{1+u^2(t)}=- \mathrm{Re}\big(i\sum_{n=1}^\infty a_n r^n e^{int}\big) 
$$
for every $t\in\mathbb R$.

For any $t\in \mathbb R$, by taking the integral $\int_0^t$ of both sides of the above equation we obtain
\begin{equation}
\begin{split}
\arctan{u(t)}-\arctan{u(0)}&=- \mathrm{Re}\Big(i \sum_{n=1}^\infty a_n r^n \frac{e^{int}-1}{i n}\Big) \\
&=-\mathrm{Re}\Big(\sum_{n=1}^\infty  a_n r^n \frac{e^{int}-1}{n}\Big),
\end{split}
\end{equation}
and therefore 
\begin{equation*}
\begin{split}
u(t)&=\tan \Big[\arctan{u(0)}- \mathrm{Re}\Big(\sum_{n=1}^\infty a_n r^n \frac{e^{int}-1}{n}\Big)\Big]\\
      &=\tan \Big[\arctan{Q_0(r)}- \mathrm{Re}\Big(\sum_{n=1}^\infty a_n r^n \frac{e^{int}-1}{n}\Big)\Big]. 
\end{split}
\end{equation*}
Thus any solution of the differential equation $\mathrm{(i)}$ has a form as
\begin{equation*}
\begin{split}
 Q_0(z_2) =\tan \Big[q(|z_2|)- \mathrm{Re}\Big(\sum_{n=1}^\infty \frac{a_n}{n} z_2^n\Big)\Big], 
\end{split}
\end{equation*}
where $q$ is a $\mathcal{C}^1$-smooth real-valued function $[0,\epsilon_0 )$, as desired. 

Next, we shall solve the differential equation $\mathrm{(ii)}$. Indeed, from Eq. $\mathrm{(ii)}$ we have
$$
2\mathrm{Re}\Big(i\beta z_2 \frac{{P_1}_{z_2}(z_2)}{P_1(z_2)}\Big)=\mathrm{Re}\big( a_1(z_2)\big) +Q_0(z_2)\mathrm{Re}\big(\frac{a_1(z_2)}{i}\big)
$$
for every $z_2\in \Delta_{\epsilon_0}^*$. In order to find a solution of the above equation, for a fixed positive number $0<r<\epsilon_0$, again let $u(t)=\log |P(re^{it})|$ for all $t\in \mathbb R$. Then one obtains that
\begin{equation*}
\begin{split}
 u'(t)&=\mathrm{Re} \Big(\sum_{n=1}^\infty a_n r^n e^{int}\Big)+ 
Q_0(re^{it}) \mathrm{Re}\Big(\sum_{n=1}^\infty \frac{a_n}{i} r^n e^{int}\Big)\\
&=\mathrm{Re} \Big(\sum_{n=1}^\infty a_n r^n e^{int}\Big)+ 
\mathrm{Re}\Big(\sum_{n=1}^\infty \frac{a_n}{i} r^n e^{int}\Big) \\
&\quad \times\tan \Big[q(r)- 
\mathrm{Re}\Big(\sum_{n=1}^\infty \frac{a_n}{n} (r^ne^{int}-r^n)\Big)\Big]
\end{split}
\end{equation*}
for all $t\in \mathbb R$. Therefore, by taking the integral $\int_0^t$ of both sides of the above equation, any solution of Eq. $\mathrm{(ii)}$ has a form as
\begin{equation*}
\begin{split}
P_1(z_2)&=\exp\Big[ p(|z_2|)+\mathrm{Re}\Big(\sum_{n=1}^\infty  \frac{a_n}{in}z_2^n\Big ) -\log \Big|\cos\Big(q(|z_2|)- \mathrm{Re}\big(\sum_{n=1}^\infty\frac{a_n}{n} z_2^n\big)\Big )\Big|\Big]
\end{split}
\end{equation*}
for all $z_2 \in\Delta_{\epsilon_0}^*$, where $p$ is a $\mathcal{C}^1$-smooth function on $(0,\epsilon_0)$ and is chosen so that $P_1(z)$ is $\mathcal{C}^1$-smooth on $\Delta_{\epsilon_0}$, as desired. 

Finally, using the same argument as the above we conclude from Eq. $\mathrm{(iii)}$ that
$$
P(z_2)=\frac{1}{\alpha}\log\Big[1+P_1(z_2)\Big]
$$
for all $z_2\in\Delta_{\epsilon_0}^*$. Thus, the proof is complete.
\end{proof}
\begin{lemma}\label{pde2} Suppose that $Q_0,Q_1$ are real-analytic functions on $\Delta_{\epsilon_0}~(\epsilon_0>0)$ with $Q_0(0)=0$ and $a_1,a_2$ are holomorphic functions on $\Delta_{\epsilon_0}$ with $a_1(0)=a_2(0)=0$ and $ \nu_0(a_1)<+\infty$ satisfying the following equations:
\begin{itemize}
\item[(a)] $\mathrm{Re}\Big[2i\beta  z_2 {Q_0}_{z_2}(z_2)+i a_1(z_2)\Big(1+Q_0^2(z_2)\Big)\Big]\equiv 0$;%
\item[(b)] $\mathrm{Re}\Big[i \Big(1+Q_0^2(z_2)\Big)a_2(z_2)- i Q_1(z_2) a_1(z_2) \Big]\equiv 0;$
\item[(c)] $
\mathrm{Re}\Big[i\beta z_2(Q_1)_{z_2}(z_2)+\frac{Q_1(z_2)}{i} \big(i-Q_0(z_2)\big) a_1(z_2)+\Big(\frac{1}{2}+\frac{Q_0(z_2)}{2i}\Big)\\
 ~~~~\times\Big(-Q_1(z_2) a_1(z_2)+\big(i-Q_0(z_2)\big)^2a_2(z_2)\Big)\Big]\equiv 0 $
\end{itemize}
on $\Delta_{\epsilon_0}$. Then we obtain that
\begin{equation*}
\begin{split}
Q_1(z_2)\equiv Q_1(0)\Big(1+Q_0^2(z_2)\Big)~\text{and}~ a_2(z_2)\equiv Q_1(0)a_1(z_2).\end{split}
\end{equation*}
\end{lemma}
\begin{proof}
The proof will be divided into two following cases.
\smallskip

\noindent
{\bf Case (i).}{ \boldmath $Q_1(0)=0$.} In this case, we will show that $Q_1\equiv 0$, and thus $a_2\equiv 0$. Indeed, suppose that, contrary to our claim, $Q_1\not \equiv 0$. Then by $\mathrm{(b)}$ we get $\nu_0(a_2)=\nu_0(Q_1)+\nu_0(a_1)> \nu_0(Q_1)$, and moreover $Q_1$ cannot contain non-harmonic terms of degree $\nu_0(Q_1)$. However, it follows from $\mathrm{(c)}$ that 
$\nu_0(Q_1)=\nu_0(a_2)$, which is a contradiction. Therefore, $Q_1\equiv 0$ and $a_2\equiv 0$.

\smallskip
\noindent
{\bf Case (ii).}{ \boldmath $Q_1(0)\ne 0$.} Let $\tilde Q_1(z_2):=Q_1(z_2)-Q_1(0)$ and $\tilde a_2(z_2):=a_2(z_2)-Q_1(0)a_1(z_2)$ for all $z_2\in \Delta_{\epsilon_0}$. Then the equation $\mathrm{(c)}$ is equivalent to
\begin{equation}\label{q4}
\begin{split}
&\mathrm{Re}\Big\{i\beta z_2(\tilde Q_1)_{z_2}(z_2)+\frac{1}{2}Q_1(z_2) a_1(z_2)-\frac{3}{2i} Q_1(z_2)Q_0(z_2) a_1(z_2) \\ 
&\quad-\frac{a_2(z_2)}{2} -\frac{i}{2}Q_0(z_2) a_2(z_2)-\frac{1}{2}Q_0^2(z_2)a_2(z_2)-\frac{i}{2}Q_0^3(z_2) a_2(z_2)\Big\}\\
&= \mathrm{Re}\Big\{i\beta z_2(\tilde Q_1)_{z_2}(z_2) -\frac{1}{i} Q_1(0)Q_0(z_2) a_1(z_2)-\frac{1}{2}Q_0^2(z_2)Q_1(0)a_1(z_2)\\
&\quad-\frac{i}{2}Q_0^3(z_2)Q_1(0) a_1(z_2) +\tilde Q_1(z_2) \Big[ \frac{a_1(z_2)}{2}-\frac{3}{2i}Q_0(z_2)a_1(z_2)\Big]\\
&\quad +\tilde a_2(z_2)\Big[-\frac{1}{2}-\frac{i}{2}Q_0(z_2)-\frac{1}{2} Q_0^2(z_2)-\frac{i}{2}Q_0^3(z_2)\Big]\Big\}\equiv 0
\end{split}
\end{equation}
on $\Delta_{\epsilon_0}$. Moreover, the equation $\mathrm{(b)}$ is equivalent to
\begin{equation*}
\mathrm{Re}\Big[i \Big(1+Q_0^2(z_2)\Big)\tilde a_2(z_2)+i\Big( Q^2_0(z_2)Q_1(0)-\tilde Q_1(z_2)\Big)a_1(z_2) \Big]\equiv 0,
\end{equation*}
or equivalently
\begin{equation}\label{q5}
\mathrm{Re}\Big[i \Big(1+Q_0^2(z_2)\Big)\tilde a_2(z_2)-iR_1(z_2)a_1(z_2) \Big]\equiv 0
\end{equation}
on $\Delta_{\epsilon_0}$, where $R_1(z_2):=\tilde Q_1(z_2)-Q_0^2(z_2) Q_1(0)$, for simplicity. By $\mathrm{(a)}$ and by a simple computation, we get
$$
 \mathrm{Re}\Big\{i\beta z_2(Q_0^2(z_2))_{z_2}-Q_0(z_2)\frac{a_1(z_2)}{i}-Q_0^3(z_2)\frac{a_1(z_2)}{i}\Big\}\equiv 0
$$
on $\Delta_{\epsilon_0}$. Hence, it follows from the above equation and (\ref{q4}) that
\begin{equation}\label{q6}
\begin{split}
&\mathrm{Re}\Big\{i\beta z_2(R_1)_{z_2}(z_2)-\frac{1}{2}Q_0^2(z_2)Q_1(0)a_1(z_2)+\frac{3}{2i}Q_0^3(z_2)Q_1(0) a_1(z_2)\\
&\quad  + Q_1(0)Q_0^2(z_2) \Big[ \frac{a_1(z_2)}{2}-\frac{3}{2i}Q_0(z_2)a_1(z_2)\Big]\\
&\quad+R_1(z_2) \Big[ \frac{a_1(z_2)}{2}-\frac{3}{2i}Q_0(z_2)a_1(z_2)\Big]\\
&\quad+\tilde a_2(z_2)\Big[-\frac{1}{2}-\frac{i}{2}Q_0(z_2)-\frac{1}{2} Q_0^2(z_2)-\frac{i}{2}Q_0^3(z_2)\Big]\\
&=\mathrm{Re}\Big\{i\beta z_2(R_1)_{z_2}(z_2)+R_1(z_2) \Big[ \frac{a_1(z_2)}{2}-\frac{3}{2i}Q_0(z_2)a_1(z_2)\Big]\\
&\quad +\tilde a_2(z_2)\Big[-\frac{1}{2}-\frac{i}{2}Q_0(z_2)-\frac{1}{2} Q_0^2(z_2)-\frac{i}{2}Q_0^3(z_2)\Big]\Big\}\equiv 0
\end{split}
\end{equation}
on $\Delta_{\epsilon_0}$.

Finally, since $R_1(0)=0$, by the same argument as in Case $\mathrm{(i)}$ with $\mathrm{(b)}$ and $\mathrm{(c)}$ replaced by 
(\ref{q5}) and (\ref{q6}) respectively, we establish that $R_1\equiv 0$ and $\tilde a_2\equiv 0$. Hence, $a_2(z_2)\equiv Q_1(0) a_1(z_2)$ and $Q_1(z_2)\equiv Q_1(0)\Big(1+Q_0^2(z_2)\Big)$ on $\Delta_{\epsilon_0}$, and the proof is thus complete.
\end{proof}

\begin{lemma}\label{pde3} Let $F(z_2,t)$ be a function defined on a neighborhood $U\times I$ of $0\in \mathbb C\times \mathbb R$ with $F(z_2,0)\equiv 0$ such that $F,\frac{\partial F}{\partial t}$, and $\frac{\partial^2 F }{\partial t^2}$ are $\mathcal{C}^1$-smooth on $U\times I$ and let $\alpha \in \mathbb R$. Then 
$$
\mathrm{Re}\Big\{\Big[\big(i+\frac{\partial F}{\partial t}(z_2,t)\big)\exp\Big(\alpha\big(it-F(z_2,t)\big)\Big)-\big(i+\frac{\partial F }{\partial t}(z_2,0)\big)\Big] a(z_2) \Big\}\equiv 0 ~\text{on}~U\times I,
$$
where $a(z_2)$ is a non-zero holomorphic function on $U$ with $a(0)=0$, 
if and only if
\[
 F(z_2,t)=\begin{cases}
 -\frac{1}{\alpha}\log \Big|\frac{\cos \big(R(z_2)+\alpha t\big)}{\cos (R(z_2))} \Big| &~\text{if}~ \alpha\ne 0\\
 \tan(R(z_2))t  &~\text{if}~ \alpha = 0
\end{cases} 
 \]
 for all $(z_2,t)\in U\times I$, where $R$ is a $\mathcal{C}^1$-smooth function on $U$.

\end{lemma}
\begin{proof}
It is not hard to check that
$$
\mathrm{Re}\Big\{\Big[\frac{\partial}{\partial t}\Big(\big(i+\frac{\partial F }{\partial t}(z_2,t)\big)\exp\big(\alpha(it-F(z_2,t))\big)\Big)\Big] a(z_2)\Big\}\equiv 0~\text{on}~U\times I 
$$
if and only if
\begin{equation}\label{pt1}
\begin{split}
\mathrm{Re}\Big\{\Big[\frac{\partial^2 F}{\partial t^2}(z_2,t)-\alpha \Big(1+\big(\frac{\partial F}{\partial t}(z_2,t)\big)^2\Big)\Big]a(z_2)\Big\}\equiv 0~\text{on}~ U\times I.
\end{split}
\end{equation}
On the other hand, we have 
\begin{equation*}
\begin{split}
&\mathrm{Re}\Big\{\Big[\frac{\partial^2 F}{\partial t^2}(z_2,t)-\alpha \Big(1+\big(\frac{\partial F}{\partial t}(z_2,t)\big)^2\Big)\Big]a(z_2)\Big\}\\
&\equiv \Big[\frac{\partial^2 F}{\partial t^2}(z_2,t)-\alpha \Big(1+\big(\frac{\partial F}{\partial t}(z_2,t)\big)^2\Big)\Big]\mathrm{Re}\big(a(z_2)\big).
\end{split}
\end{equation*}
on $U\times I$. Since $\mathrm{Re}\big(a(z_2)\big)\not \equiv 0$ on $U$, Eq. (\ref{pt1}) is equivalent to
$$
\frac{\partial^2 F}{\partial t^2}(z_2,t)\equiv \alpha \Big(1+\big(\frac{\partial F}{\partial t}(z_2,t)\big)^2\Big)~\text{on}~U\times I.
$$
Moreover, it follows from the above equation that
$$
\frac{\partial F}{\partial t}(z_2,t)=\tan(R(z_2)+\alpha t)
$$
for all $(z_2,t)\in U\times I$. Hence, the function $F$ has the form as in the lemma.
\end{proof}
\begin{corollary}\label{pde4} Let $\epsilon_0,\beta,\alpha \in \mathbb R$ with $\beta\ne 0 $ and $\epsilon_0>0$. Suppose that $R:\Delta_{\epsilon_0}\to [-1,1]$ is $\mathcal{C}^1$-smooth  satisfying 
$$
2\mathrm{Re}\Big(i\beta z_2 R_{z_2}(z_2)\Big)=- \mathrm{Re}\big(ia_1(z_2)\big)
$$
 for all $z_2\in \Delta_{\epsilon_0}$, where $a_1(z_2)$ is a non-zero holomorphic function defined on $\Delta_{\epsilon_0}$. Let $F(z_2,t): \Delta_{\epsilon_0}\times (-\delta_0,\delta_0)\to \mathbb R $ be a function defined by
 \[
 F(z_2,t)=\begin{cases}
 -\frac{1}{\alpha}\log \Big|\frac{\cos \big(R(z_2)+\alpha t\big)}{\cos (R(z_2))} \Big| &~\text{if}~ \alpha\ne 0\\
 \tan(R(z_2))t  &~\text{if}~ \alpha = 0,
\end{cases} 
 \]
 where $\delta_0=\frac{1}{2|\alpha|}$ if $\alpha\ne 0$ and $\delta_0=+\infty$ if otherwise. Then we have
  \begin{equation}\label{3042013}
 \mathrm{Re}\Big[2i\alpha\beta  z_2 F_{z_2}(z_2,t)+\Big(F_t(z_2,t)-\tan\big(R(z_2)\big)\Big) ia_1(z_2)\Big]= 0
\end{equation}
 for all $(z_2,t)\in \Delta_{\epsilon_0}\times (-\delta_0,\delta_0)$.
\end{corollary}
\begin{proof}
By a direct computation we obtain that
\begin{equation*}
\begin{split}
F_{z_2}(z_2,t)= 
\begin{cases}
\frac{1}{\alpha}\Big( \tan \big(R(z_2)+\alpha t\big)-\tan \big(R(z_2)\big)\Big) R_{z_2}(z_2)&~\text{if}~\alpha\ne 0\\
\Big(1+\tan^2\big(R(z_2)\big)\Big)R_{z_2}(z_2)t&~\text{if}~\alpha= 0
\end{cases}
\end{split}
\end{equation*}
and
\begin{equation*}
\begin{split}
F_{t}(z_2,t)= 
\begin{cases}
 \tan \big(R(z_2)+\alpha t\big) &~\text{if}~\alpha\ne 0\\
\tan(R(z_2)) &~\text{if}~\alpha= 0
\end{cases}
\end{split}
\end{equation*}
for all $(z_2,t)\in \Delta_{\epsilon_0}\times (-\delta_0,\delta_0)$.

If $\alpha=0$, then (\ref{3042013}) is trivial. So, we only consider the case $\alpha\ne 0$. By our assumption, we thus obtain that 
\begin{equation*}
\begin{split}
 \mathrm{Re}\Big[2i\alpha\beta  z_2 F_{z_2}(z_2,t)\Big]&=\Big( \tan \big(R(z_2)+\alpha t\big)-\tan \big(R(z_2)\big)\Big) \mathrm{Re}\big(-ia_1(z_2)\big)\\
 &=\Big(F_t(z_2,t)-\tan\big(R(z_2)\big)\Big)\mathrm{Re}\big(- ia_1(z_2)\big)
\end{split}
\end{equation*}
for all  $(z_2,t)\in \Delta_{\epsilon_0}\times (-\delta_0,\delta_0)$. Therefore, Eq. (\ref{3042013}) holds, and thus which ends the proof.
\end{proof}

\subsection*{A.3. Several technical lemmas}
In what follows $P$ stands for a real $\mathcal{C}^\infty$-smooth function defined on the disc $\Delta_{\epsilon_0}:=\{z\in \mathbb C\colon |z|< \epsilon_0\}$ with sufficiently small radius $\epsilon_0>0$ satisfying $P(0)=0$, $P(z)>0$ for any $z\in \Delta_{\epsilon_0}^*:=\Delta_{\epsilon_0}\setminus\{0\}$, and $\nu_0(P)=+\infty$. 

First of all, we recall the following lemma which is a slight generalization of \cite[Lemma 3]{Kim-Ninh} and it is proved in \cite{NCM} .
\begin{lemma}[see Lemma $1$ in \cite{NCM} or Lemma $3$ in \cite{Kim-Ninh}]\label{Al3}
If $a, b$ are complex numbers and if $g_0, g_1, g_2$ are $\mathcal{C}^\infty$-smooth functions defined on the disc $\Delta_{\epsilon_0}$ satisfying:
\begin{itemize}
\item[(A1)] $g_0(z) = O(|z|)$, $g_1(z) = O(|z|^\ell)$, and $g_2(z) = o(|z|^m)$, and
\item[(A2)] $\mathrm{Re} \Big[a z^m+\frac{1}{P^n(z)}\Big(b z^\ell\big(1+g_0(z)\big) \frac{P'(z)}{P(z)}
+g_1(z) \Big)\Big]= g_2(z)$ for every $z \in \Delta^*_{\epsilon_0}$
\end{itemize}
for any nonnegative integers $\ell, m$ and $n$, except for the following two cases
\begin{itemize}
\item[(E1)] $\ell=1$ and $\mathrm{Re}(b) = 0$, and
\item[(E2)] $m=0$ and $\mathrm{Re}(a) = 0$,
\end{itemize}
then $ab=0$.
\end{lemma}

Following the proof of Lemma $1$ in \cite{NCM}, we have the following corollary.
\begin{corollary}\label{cor2}
If $b$ is a complex number and if $g_0, g_1, g_2$ are $\mathcal{C}^\infty$-smooth functions defined on the 
disc $\Delta_{\epsilon_0}$ with sufficiently small radius satisfying:
\begin{itemize}
\item[(A1)] $g_0(z) = O(|z|)$, $g_1(z) = O(|z|^\ell)$, and $\nu_0(g_2)=m$ or $\nu_0(g_2)=+\infty$, and
\item[(A2)] $\mathrm{Re} \Big(b z^\ell\big(1+g_0(z)\big) \frac{P'(z)}{P(z)}
+g_1(z) \Big)= g_2(z)$ for every $z \in \Delta^*_{\epsilon_0}$
\end{itemize}
for some nonnegative integer $\ell$ and for some positive integer $m$, except for the following two cases
\begin{itemize}
\item[(E1)] $\ell=1$ and $\mathrm{Re}(b) = 0$, and
\item[(E2)] $\ell\geq 2$ and $0<m/(\ell-1)<1$,
\end{itemize}
then $b=0$.
\end{corollary}

Let $F$ be a $\mathcal{C}^1$-smooth complex-valued function defined in a neighborhood $U$ of the origin in the complex plane.  Consider the autonomous dynamical system 
\begin{equation}\label{eq??1}
\frac{dz}{dt}=F(z), z(0)=z_0\in U.
\end{equation}
Let us now recall several definitions.
\begin{define} A state $\hat z\in U$ is called an equilibrium of (\ref{eq??1}) if $F(\hat z)=0$.
\end{define}

\begin{define} An equilibrium, $\hat z$, of (\ref{eq??1}) is called locally asymptotically stable if for all $\epsilon>0$ there exists $\delta>0$ such that $|z_0-\hat z|<\delta$ implies that $|z(t)-\hat z|<\epsilon$ for all $t\geq 0$ and $\lim_{t\to +\infty} z(t)=0$. 
\end{define}
\begin{lemma}\label{l4} Let $b \in \mathbb C$ with $\mathrm{Re}(b)< 0$ and let $f: \Delta_{\epsilon}\to \mathbb R^*$ be a nonnegative $\mathcal{C}^1$-smooth function satisfyfing that $f(0)=0$ and $f(z)>0$ for all $z\in \Delta^*_{\epsilon}$. Then the origin is a locally asymptotically stable equilibrium of the following equation
\begin{equation}\label{eq?3}
\frac{dz}{dt}=z\Big[i\alpha+b \big(1+g(z)\big)f(z)\Big],
\end{equation}
where $\alpha\in \mathbb R^*$ and $g$ is a $\mathcal{C}^1$-smooth function defined on $\Delta_{\epsilon}$ satisfying $g(0)=0$.
\end{lemma}

\begin{proof} First of all, denote by $F(z):=z\Big[i\alpha+b \big(1+g(z)f(z)\big)\Big]$ for all $z\in \Delta_{\epsilon_0}$. Let $V(z):=\frac{1}{2}|z|^2$ for all $z\in \mathbb C$. Then it is easily checked that
\begin{itemize}
\item[(i)] $V(0)=0$,
\item[(ii)] $V(z)>0$ for all $z\ne 0$.
\end{itemize}
Moreover, by assumption there exists a neighborhood $U\subset \Delta_{\epsilon_0}$ of $0$ such that we have
\begin{equation}
 \nabla V(z). F(z)=\mathrm{Re}\big(F(z)\bar z\big)=|z|^2\mathrm{Re}\Big(b(1+g(z))f(z)\Big)<0
\end{equation}
for all $z\in U\setminus\{0\}$. Theorefore, $V$ is a strong Lyapunov function and hence by the Lyapunov's stability theorem the origin is locally asymptotically stable (cf. \cite[Theorem 10.7]{JS}).
\end{proof}
\begin{lemma}\label{Al5}
 Let $Q(z_2,t)$ be a $\mathcal{C}^\infty$-smooth funtion defined on a neighborhood of $(0,0)$ in $\mathbb C\times \mathbb R$ satisfying that $Q_{z_2}(z_2,0)$ is real-analytic and let $h_2$ be a non-zero holomorphic function defined on a neighborhood of $(0,0)$ in $\mathbb C^2$. If
\begin{equation}\label{eq?4}
\mathrm{Re}\Big[\Big(P'(z_2)+tQ_{z_2}(z_2,t)\Big)h_2\big(it-P(z_2)-tQ(z_2,t),z_2\big)\Big]=0
\end{equation}
for all $(z_2,t)$ in a neighborhood of $(0,0)$ in $ \mathbb C\times\mathbb R$, then, after a change of variable in $z_2$, $h_2(z_1,z_2)\equiv i z_2$ and $P,Q$ are rotational in $z_2$, i.e., $P(z_2)=P(|z_2|)$ and $Q(z_2,t)=Q(|z_2|,t)$.
\end{lemma}
\begin{proof}
Expand $h_2$ into the Taylor series at the origin so that
$$
h_2(z_1,z_2)=\sum_{n=0}^\infty a_n(z_2)z_1^n,
$$
where $a_n$ is holormorphic in a neighborhood of $0$ in $\mathbb C$ for all $n\in \mathbb N$. Then (\ref{eq?4}) is equivalent to 
\begin{equation}\label{eq?5}
\mathrm{Re}\Big[\Big(P'(z_2)+tQ_{z_2}(z_2,t)\Big)\sum_{n=0}^\infty \big(it-P(z_2)-tQ(z_2,t)\big)^n a_n(z_2)\Big]=0
\end{equation}
for all $(z_2,t)\in \Delta_{\epsilon_0}\times(-\delta_0,\delta_0)$, where $\epsilon_0>0$ and $\delta_0>0$ are small enough.

Since $h_2\not \equiv 0$, there is the smallest $n_0$ such that $a_{n_0}\not\equiv 0$ and thus it can be written as follows:
$$ 
a_{n_0}(z_2)=a_{n_0 m_0} z_2^{m_0}+o(z_2^{m_0}),
$$
where $m_0=\nu_0(a_{n_0})$ and $a_{n_0 m_0}\in \mathbb C^*$. Moreover, since $P(z_2)=o(|z_2|^{m_0})$ and $Q(0,0)=0$, it follows from (\ref{eq?5}) with $t=\alpha P(z_2)$ ($\alpha\in \mathbb R$ will be chosen later) that
$$
\mathrm{Re}\Big[\big(i\alpha-1 \big)^{n_0} \big(a_{n_0 m_0}z_2^{m_0}+o(|z_2|^{m_0})\big)\frac{P'(z_2)}{P(z_2)}\Big]=g(z_2)
$$
for every $z_2\in\Delta^*_{\epsilon_0}$, where $g$ is the function defined on $\Delta_{\epsilon_0}$ by setting $g(z_2)=-\mathrm{Re}\Big[\alpha \big(i\alpha-1 \big)^{n_0} Q_{z_2}(z_2,\alpha P(z_2))\big(a_{n_0 m_0}z_2^{m_0}+o(|z_2|^{m_0}\big)\Big ]$.

Notice that if $n_0>0$ then we can choose $\alpha$ so that $\mathrm{Re}\Big[ a\big(i\alpha-1 \big)^{n_0}\Big]\ne 0$. Therefore, since $\nu_0(g)\geq m_0$ it follows from Corollary \ref{cor2} that $n_0=0, m_0=1$, and $\mathrm{Re}(a_{n_0 m_0})=0$. By a change of variable in $z_2$ (cf. Lemma \ref{lemma6}), we can assume that $a_0(z_2)\equiv iz_2$. 

Next, we shall prove that $a_k\equiv 0$ for every $k=1,2,\ldots$. Indeed, suppose otherwise. Then let $k_0>0$ be the smallest integer such that $a_{k_0}\not \equiv  0$. Thus it can be written as follows:
$$ 
a_{k_0}(z_2)=a_{k_0 m_0} z_2^{m_0}+o(z_2^{m_0})
$$
where $m_0=\nu_0(a_{k_0})$ and $a_{k_0 m_0}\in \mathbb C^*$. Taking $\frac{\partial}{\partial t}$ at $t=0$ of both sides of the equation (\ref{eq?5}) and notice that $P(z_2)=o(|z_2|^{m_0})$, one obtains that
\begin{equation}\label{eq?6}
\begin{split}
& \mathrm{Re}\Big[ik_0\big(-P(z_2)\big)^{k_0-1}\big(a_{k_0 m_0} z_2^{m_0}+o(|z_2|^{m_0})\big) P'(z_2)\\
&+ Q_{z_2}(z_2,0)\Big(iz_2+\big(-P(z_2)\big)^{k_0}\big(a_{k_0 m_0} z_2^{m_0}+o(|z_2|^{m_0})\big)\Big)\Big]=0
\end{split}
\end{equation}
for all $z_2\in  \Delta_{\epsilon_0}$. 

Since $Q_{z_2}(z_2,0)$ is real-analytic and $\nu_0(P)=\nu_0(P')=+\infty$, $ \mathrm{Re}\Big[Q_{z_2}(z_2,0)\Big]\equiv 0$ on $\Delta_{\epsilon_0}$ and hence $Q(z_2,0)$ is rotational (cf. \cite[Lemma 4]{Kim-Ninh}). Therefore, we arrive at 
\begin{equation}\label{eq?7}
\begin{split}
& \mathrm{Re}\Big[ik_0\big(a_{k_0 m_0} z_2^{m_0}+o(|z_2|^{m_0})\big) \frac{P'(z_2)}{P(z_2)}\\
&- Q_{z_2}(z_2,0)\Big(a_{k_0 m_0} z_2^{m_0}+o(|z_2|^{m_0})\Big)\Big]=0
\end{split}
\end{equation}
for all $z_2\in  \Delta_{\epsilon_0}$. Following the argument as above, by Corollary \ref{cor2} we conclude that
$a_{k_0}(z_2)\equiv \beta z_2(1+O(z_2))$, where $\beta\in \mathbb R^*$. Without loss of generality we may assume that $\beta<0$. Thus, since $\nu_0(P)=+\infty$, inserting $t=0$ into (\ref{eq?5}) one has 
\begin{equation}\label{eq?6}
\mathrm{Re}\Big[ z_2\Big( i +\beta \big(1+O(|z_2|)\big)\Big)P'(z_2)\Big]\equiv 0
\end{equation}
on $\Delta_{\epsilon_0}$. So, Lemma \ref{l4} tells us that, with no loss of generality, 
there exists a flow $\gamma: [t_0,+\infty)\to \Delta^*_{\epsilon_0}~(t_0>0)$ of the following equation
$$
 \frac{dz_2}{dt}=z_2\Big( i +\beta\big(1+O(|z_2|)\big)\Big)
$$
satisfying $\gamma(t)\to 0$ as $t\to +\infty$. Hence, by (\ref{eq?6}) one gets
$\frac{d P(\gamma(t))}{dt}\equiv 0$, and therefore $P(\gamma(t))\equiv \lim_{t\to +\infty} P(\gamma(t))=P(0)=0$, which is absurd.
This proves that $h_2(z_1,z_2)\equiv iz_2$. 

Consequently, (\ref{eq?4}) is now equivalent to
\begin{equation}\label{eq?17}
\mathrm{Re}\Big[iz_2\Big(P'(z_2)+tQ_{z_2}(z_2,t)\Big)\Big]=0
\end{equation}
for all $(z_2,t)\in \Delta_{\epsilon_0}\times (-\delta_0,\delta_0)$. This implies that $\mathrm{Re}\Big[iz_2 P'(z_2)\Big]\equiv 0$ on $\Delta_{\epsilon_0}$, and thus it follows from \cite[Lemma 4]{Kim-Ninh} that $P$ is rotational. Furthermore, one has by (\ref{eq?17})
$$
\mathrm{Re}\Big[iz_2Q_{z_2}(z_2,t)\Big]=0
$$
for all $(z_2,t)\in \mathbb C\times (-\delta_0,\delta_0)$. Again by \cite[Lemma 4]{Kim-Ninh}, this in turn yields $Q(z_2,t)$ is rotational in $z_2$. This ends the proof.
\end{proof}

\begin{Acknowlegement} The author would like to thank Prof. Kang-Tae Kim, Prof. Do Duc Thai, and Dr. Hyeseon Kim for their precious discussions on this material. It is a pleasure to thank Prof. Nguyen Quang Dieu for his helpful suggestions. 
\end{Acknowlegement}


\begin{thebibliography}{99}
\bibitem{B-P1}
E. Bedford, S. Pinchuk, Domains in $\mathbb C^2$ with noncompact groups of
automorphisms, Math. USSR Sbornik 63 (1989), 141--151.

\bibitem{B-P2}
E. Bedford, S. Pinchuk, Domains in $\mathbb C^{n+1}$ with noncompact automorphism
group, J. Geom. Anal. 1 (1991), 165--191.

\bibitem{B-P3}
E. Bedford, S. Pinchuk, Domains in $\mathbb C^2$
with noncompact automorphism groups, Indiana Univ. Math. J. 47 (1998), 199-222.
\bibitem{Bao} M. S. Baouendi, P. Ebenfelt, L. P. Rothschild, Real submanifolds in complex space and their mappings, Princeton Mathematical Series, 47. Princeton University Press, Princeton, NJ, 1999.


\bibitem{Bell-Lig} S. Bell, E. Ligocka, A simplification and extension of
Fefferman's theorem on biholomorphic mappings, Invent. Math. 57 (3) (1980), 283--289.

\bibitem{Ber} F. Berteloot, Characterization of models in $\mathbb C^2$
by their automorphism groups, Internat. J. Math. 5 (1994), 619--634.


\bibitem{By1} J. Byun, J.-C. Joo, M. Song, The characterization
of holomorphic vector fields vanishing at an infinite type point, J. Math. Anal. Appl. 387 (2012), 667--675.
\bibitem{Co} C. Coleman, Equivalence of planar dynamical and differential systems, J. Differential Equations 1 (1965), 222--233.
\bibitem{D} J. P. D'Angelo, Real hypersurfaces, orders of contact, and applications, Ann. Math. 115 (1982), 615--637.

\bibitem{Fef} C. Fefferman, The Bergman kernel and biholomorphic mappings
of pseudoconvex domains, Invent. Math. 26 (1974), 1--65.


\bibitem{GGJ} A. Garijo, A. Gasull, X. Jarque, Local and global phase portrait of equation $\dot z=f(z)$, Discrete Contin. Dyn. Syst. 17 (2)( 2007), 309--329.

\bibitem{GK} R. Greene, S. G. Krantz, Techniques for studying
automorphisms of weakly pseudoconvex domains, Math. Notes,
Vol 38, Princeton Univ. Press, Princeton, NJ, 1993, 389--410.



\bibitem{IK} A. Isaev, S. G. Krantz, Domains with non-compact automorphism group:
A survey, Adv. Math. 146 (1999), 1--38.
\bibitem{JS} D. W. Jordan, P. Smith, Nonlinear ordinary differential equations. An introduction for scientists and engineers,  Fourth edition, Oxford University Press, Oxford, 2007.









\bibitem{Kim-Ninh} K.-T. Kim, V. T. Ninh, On the tangential holomorphic vector fields vanishing at an infinite type point, http://arxiv.org/abs/1206.4132, Trans. Amer. Math. Soc., to appear.



\bibitem{NCM} V. T. Ninh, V. T. Chu, A. D. Mai, On the real-analytic infinitesimal CR automorphism of hypersurfaces of infinite type, http://arxiv.org/abs/1404.4914.
\bibitem{Ninh1} V. T. Ninh, On the CR automorphism group of a certain hypersurface of infinite type in $\mathbb C^2$, http://arxiv.org/abs/1311.3050.

\bibitem{R} J.-P. Rosay, Sur une caracterisation de la boule parmi les domaines
de $\mathbb C^n$ par son groupe d'automorphismes, Ann. Inst. Fourier 29 (4) (1979), 91--97.


\bibitem{S} R. Sverdlove, Vector fields defined by complex functions, J. Differential Equations 34 (3) (1979), 427--439. 

\bibitem{W} B. Wong, Characterization of the ball in $\mathbb C^n$ by its automorphism group, Invent. Math. 41 (1977), 253--257.
\end{thebibliography}
\end{document}